\documentclass[10pt]{article}
\usepackage{hyperref}
\hypersetup{hypertexnames = false, bookmarksdepth = 2, bookmarksopen = true, colorlinks, linkcolor = black, citecolor = black, urlcolor = black, pdfstartview={XYZ null null 1}}

\usepackage{amsfonts}
\usepackage[fleqn, leqno]{amsmath}
\usepackage{amsthm}
\usepackage{amssymb}
\usepackage[noabbrev]{cleveref}

\usepackage[maxbibnames=99, sortcites]{biblatex}
\usepackage{booktabs}
\usepackage{diagbox}
\usepackage{enumitem}
\usepackage{mathtools}
\usepackage{parskip}
\usepackage{subcaption}
\usepackage{thmtools}
\usepackage{tikz-cd}
\usepackage[colorinlistoftodos, disable, textsize = footnotesize]{todonotes}
\usepackage{xparse}
\usepackage{xspace}

\usepackage[T1]{fontenc}
\usepackage{libertine}
\usepackage[libertine]{newtxmath}
\usepackage[scaled=0.83]{beramono}
\usepackage{eucal}
\usepackage{microtype}
\usepackage{gitinfo2}

\usepackage[all,ps,cmtip]{xy}
\usepackage{mathrsfs}
\newcounter{todocounter}
\DeclareDocumentCommand\addreference{g}{\stepcounter{todocounter}\todo[color = blue!30]{\thetodocounter. Add reference\IfNoValueF{#1}{: #1}}\xspace}
\DeclareDocumentCommand\checkthis{g}{\stepcounter{todocounter}\todo[color = red!50]{\thetodocounter. Check this\IfNoValueF{#1}{: #1}}\xspace}
\DeclareDocumentCommand\fixthis{g}{\stepcounter{todocounter}\todo[color = orange!50]{\thetodocounter. Fix this\IfNoValueF{#1}{: #1}}\xspace}
\DeclareDocumentCommand\expand{g}{\stepcounter{todocounter}\todo[color = green!50]{\thetodocounter. Expand\IfNoValueF{#1}{: #1}}\xspace}

\newtheorem{theorem}{Theorem}
\newtheorem{corollary}[theorem]{Corollary}

\newtheorem{lemma}[theorem]{Lemma}
\newtheorem{proposition}[theorem]{Proposition}

\newtheorem{remark}[theorem]{Remark}

\newtheorem{alphatheorem}{Theorem}
\newtheorem{alphacorollary}[alphatheorem]{Corollary}

\crefname{alphatheorem}{theorem}{theorems}
\crefname{alphacorollary}{corollary}{corollaries}

\crefname{subfigure}{figure}{figures}

\usetikzlibrary{arrows,arrows.meta}
\tikzcdset{arrow style=tikz, diagrams={>={Classical TikZ Rightarrow[]}}}

\makeatletter
\def\gitfootnote{\gdef\@thefnmark{}\@footnotetext}
\makeatother

\hypersetup{pagebackref}

\mathchardef\mhyphen="2D
\newcommand\dash{\nobreakdash-\hspace{0pt}}

\makeatletter
\newcommand{\dashedrightarrow}[1][2pt]{%
  \settowidth{\@tempdima}{$\rightarrow$}\rightarrow
  \makebox[-\@tempdima]{\hskip-1.5ex\color{white}\rule[0.5ex]{#1}{1pt}}
  \phantom{\rightarrow}
}
\makeatother

\let\oldbigwedge\bigwedge
\renewcommand\bigwedge{\oldbigwedge\nolimits}

\newcommand\hilbtwo[1][X]{\ensuremath{{#1}^{[2]}}} 
\newcommand\hilbtwoP[1][n]{\ensuremath{\mathbb{P}^{#1[2]}}}
\newcommand\hilbtwoLone{\ensuremath{L_1^{[2]}}}
\newcommand\hilbtwoLtwo{\ensuremath{L_2^{[2]}}}
\newcommand\hilbn[2]{\ensuremath{#2^{[#1]}}} 

\newcommand\bounded{\ensuremath{\mathrm{b}}}
\newcommand\dgCat{\ensuremath{\mathrm{dgCat}}}
\newcommand\groundfield{\ensuremath{\mathbf{k}}}

\newcommand\LLL{\ensuremath{\mathbf{L}}}
\newcommand\normal{\ensuremath{\mathrm{N}}}

\newcommand\PP{\ensuremath{\mathbb{P}}}
\newcommand\RR{\ensuremath{\mathrm{R}}}
\newcommand\RRR{\ensuremath{\mathbf{R}}}
\newcommand\tangent{\ensuremath{\mathrm{T}}}
\newcommand\Var{\ensuremath{\mathrm{Var}}}

\DeclareMathOperator\Bl{Bl}

\DeclareMathOperator\derived{\mathbf{D}}

\DeclareMathOperator\fano{F}
\DeclareMathOperator\Gr{Gr}
\DeclareMathOperator\HH{H}

\DeclareMathOperator\Hom{Hom}
\DeclareMathOperator\RHom{\mathbf{R}Hom}

\DeclareMathOperator\Kzero{K_0}

\DeclareMathOperator\PGL{PGL}
\DeclareMathOperator\Pic{Pic}
\DeclareMathOperator\pr{pr}

\DeclareMathOperator\rk{rk}
\DeclareMathOperator\serre{\mathbb{S}}

\DeclareMathOperator\Sym{Sym}

\DeclareMathOperator\ZZ{Z}

\title{Derived categories of flips and cubic hypersurfaces}
\author{Pieter Belmans \and Lie Fu \and Theo Raedschelders}

\begin{document}
\maketitle

\bigskip

\begin{center}
  \emph{dedicated to the memory of Tom Nevins}
\end{center}

\bigskip

\begin{abstract}
  A classical result of Bondal--Orlov states that a standard flip in birational geometry gives rise to a fully faithful functor between derived categories of coherent sheaves. We complete their embedding into a semiorthogonal decomposition by describing the complement. As an application, we can lift the Galkin--Shinder relation in the Grothendieck ring of varieties between a smooth cubic hypersurface, its Fano variety of lines, and its Hilbert square, to a semiorthogonal decomposition.

 We also show that the Hilbert square of a cubic hypersurface of dimension at least~3 is again a Fano variety, so in particular the Fano variety of lines on a cubic hypersurface is a Fano visitor. The most interesting case is that of a cubic fourfold, where this exhibits the first higher-dimensional hyperk\"ahler variety as a Fano visitor.
\end{abstract}

\setcounter{tocdepth}{1}
%

\section{Introduction}

\subsection*{Derived categories of flips}
The (conjectural) interaction between birational geometry and derived categories is largely based on the \emph{DK-hypothesis} \cite{MR1949787}. This hypothesis predicts that~K\dash equivalent varieties~$X$ and~$Y$ are~D\dash equivalent, i.e.~their derived categories of coherent sheaves are equivalent as triangulated categories. Here K-equivalence means the existence of a diagram
\begin{equation}
  \begin{tikzcd}
    & Z \arrow[ld, swap, "f"] \arrow[rd, "g"] \\
    X & & Y
  \end{tikzcd}
\end{equation}
of birational morphisms between smooth projective varieties, such that~$f^*\mathrm{K}_X\sim_{\mathrm{lin}} g^*\mathrm{K}_Y$. If on the other hand we have a~K\dash inequality, i.e.~$f^*\mathrm{K}_X+D\sim_{\mathrm{lin}} g^*\mathrm{K}_Y$ for some effective divisor~$D$ on~$Z$, then it predicts the existence of a fully faithful functor~$\derived^\bounded(X)\hookrightarrow\derived^\bounded(Y)$.

We will study the case of a specific K-inequality, where we complete the fully faithful functor predicted by the DK-hypothesis into a semiorthogonal decomposition, i.e.~we explicitly describe the complement.
The special instance of a birational transformation we are interested in is that of a (standard) flip~$\phi\colon X\dashedrightarrow X'$. We recall the setup: let~$F$ be a smooth projective variety defined over an algebraically closed field~$\groundfield$ of characteristic zero. Let~$k$ and~$\ell$ be two positive integers. Let~$\mathcal{V}$ be a vector bundle of rank~$k+1$ on~$F$, denote~$Z\coloneqq\mathbb{P}(\mathcal{V})$ and let~$\pi\colon Z\to F$ be the associated projective bundle.

Let~$X$ be a smooth projective variety containing~$Z$ as a closed subvariety such that the restriction of the normal bundle~$\normal_{Z/X}$ to each fiber of~$\pi$ is isomorphic to~$\mathcal{O}(-1)^{\oplus\ell+1}$. In other words, $\normal_{Z/X}\simeq\mathcal{O}_{\pi}(-1)\otimes\pi^{*}\mathcal{V}'$ for some vector bundle~$\mathcal{V}'$ of rank $\ell+1$ on~$F$.

Let $\tau\colon\tilde X\to X$ be the blowup of~$X$ along the smooth center~$Z$, then the exceptional divisor~$E$ is isomorphic to~$\mathbb{P}(\mathcal{V})\times_{F}\mathbb{P}(\mathcal{V}')$ with normal bundle~$\normal_{E/\tilde X}\simeq\mathcal{O}(-1,-1)$. 
Therefore in the category of smooth algebraic spaces, we can contract~$E$ inside $\tilde X$ along the other direction: we get morphisms $E\to Z'\coloneqq\mathbb{P}(\mathcal{V}')$ and $\tilde X\to X'$. We assume that $X'$ is again a (smooth) projective variety. 
Hence the contraction of $E$ to $Z'$ amounts to blowing down $\tilde X$ to $X'$, and the normal bundle is $\normal_{Z'/X'}\simeq\mathcal{O}_{\pi'}(-1)\otimes\pi^{\prime *}\mathcal{V}$, where $\pi'\colon Z'\to F$ is the natural projection. Note that $E$ is naturally identified as the projectivization over $Z'$ of~$\normal_{Z'/X'}$. We summarize the situation in the following diagram.

\begin{equation}
  \begin{tikzcd}
    & & E \arrow[d, hook, "j"] \arrow[ddll, swap, "p"] \arrow[ddrr, "p'"] \\
    & & \tilde{X} \arrow[dl, swap, "\tau"] \arrow[dr, "\tau'"] \\
    Z \arrow[r, hook, "i"] \arrow[rrd, swap, "\pi"] & X \arrow[dashed, rr, "\phi"] & & X' & Z' \arrow[l, hook', swap, "i'"] \arrow[dll, "\pi'"] \\
    & & F
  \end{tikzcd}
  \label{equation:flip}
\end{equation}

Here the birational transform~$\phi$ is called a \emph{standard flip} (resp.~\emph{standard flop} when~$k=\ell$) while the inner triangle is its resolution; the upper two trapezoids are blowup diagrams and the outer square is cartesian. Observe that~$\pi$ and~$p'$ are~$\mathbb{P}^k$-bundles while~$\pi'$ and~$p$ are~$\mathbb{P}^\ell$-bundles. We will refer to~\eqref{equation:flip} as a \emph{standard flip diagram}.

By symmetry, we are free to assume that $k\geq\ell$. According to the general conjecture on derived categories under birational transformations discussed above, the derived category of~$X'$ is expected to be ``smaller'' than that of~$X$. Bondal--Orlov established this in the above setting of standard flips \cite[theorem~3.6]{alg-geom/9506012}.

\begin{theorem}[(Bondal--Orlov)]
  \label{thm:BO}
  Assume a standard flip diagram \eqref{equation:flip} is given, with~$k\geq\ell$. The functor
  \begin{equation}
    \RRR\tau_{*}\circ\LLL\tau^{\prime*}\colon \derived^\bounded(X')\to \derived^\bounded(X)
  \end{equation}
  is fully faithful. Moreover, if~$k=\ell$, this functor is an equivalence of triangulated categories.
\end{theorem}

The first goal of this paper is to identify the complement of~$\derived^\bounded(X')$ inside~$\derived^\bounded(X)$, thus completing \cref{thm:BO} into a semiorthogonal decomposition. This specific question was stated in \cite[remarque~4.5]{MR2296422}, and the following theorem accomplishes this.

\begin{alphatheorem}
  \label{theorem:flip}
  Assume a standard flip diagram \eqref{equation:flip} is given, with~$k>\ell$.
  \begin{enumerate}[label=(\roman*)]
    \item\label{enum:part1}
      For every integer~$m$, the functor
      \begin{equation}
        \label{equation:Phi-definition}
        \Phi_m\colon \derived^\bounded(F)\to \derived^\bounded(X):
        \mathcal{E} \mapsto i_{*}(\pi^{*}(\mathcal{E})\otimes \mathcal{O}_{\pi}(m)),
      \end{equation}
      is fully faithful.

    \item\label{enumerate:part-2}
      We have the following semiorthogonal decomposition of~$\derived^\bounded(X)$:
      \begin{equation}
        \label{eqn:goal}
        \derived^\bounded(X)=\langle\Phi_{-k+\ell}(\derived^\bounded(F)),\dots, \Phi_{-1}(\derived^\bounded(F)), \RRR\tau_*\circ\LLL\tau^{\prime*}\derived^\bounded(X')\rangle.
      \end{equation}
  \end{enumerate}
\end{alphatheorem}

In fact, part~\ref{enum:part1} of \cref{theorem:flip} follows from a more general fully faithfulness criterion for EZ-type functors which might be of independent interest, see \cref{prop:FF}.

For the application in \cref{thm:cubic} below we only need \cref{theorem:flip}\ref{enumerate:part-2} for~$\ell=1$. In fact, in this case a more direct proof is possible, which moreover gives the slightly stronger result that compares the category~$\derived^\bounded(X)$ and the components of \eqref{eqn:goal} as subcategories in~$\derived^\bounded(\tilde X)$, and not only after applying~$\RRR\tau_*$. This will be discussed in \cref{app:l1}.

Note also that formally inserting~$\ell=0$ into \cref{theorem:flip} one recovers Orlov's blowup formula \cite[theorem~4.3]{MR1208153}. In this paper we restrict to positive $\ell$ however.

More recently, \cite{MR3918435} and \cite{1811.12525v1} discusses the behavior of derived categories under flops (but not flips), and \cite{1910.06730v1} discusses the behavior of Chow groups under standard flips. Kawamata has proven analogous results in the toric and toroidal case, see \cite[theorem~6.1]{MR2280493}, \cite[theorem~1]{MR3079267} and \cite[theorem~1]{MR3454097}.

\begin{remark}
  \label{remark:mutation-sod}
  By mutating the first few terms of \eqref{eqn:goal} to the far right, which can be computed by applying the inverse of the Serre functor, we get a series of similar semiorthogonal decompositions: for any integer $0\leq m\leq k-\ell$, we have
  \begin{equation}
    \begin{aligned}
      \derived^\bounded(X)
      &=
      \langle
        \Phi_{-m}(\derived^\bounded(F)),
        \dots,
        \Phi_{-1}(\derived^\bounded(F)),
        \RRR\tau_{*}\circ\LLL\tau'^{*}(\derived^\bounded(X')), \\
        &\qquad\qquad
        \Phi_{0}(\derived^\bounded(F)),
        \dots,
        \Phi_{k-\ell-m-1}(\derived^\bounded(F))
      \rangle.
    \end{aligned}
  \end{equation}
\end{remark}

\subsection*{Smooth cubic hypersurfaces: the Galkin--Shinder--Voisin diagram}
As an illustration of the usefulness of \cref{theorem:flip}, we will consider a smooth cubic hypersurface~$Y\subset\PP^{n+1}$. Building upon insights of Galkin--Shinder \cite{1405.5154v2} and Voisin \cite{MR3646872}, we show in \cref{corollary:quadratic-fano-standard-flip} that the quadratic Fano correspondence
\begin{equation}
  \begin{tikzcd}
    P_2 \arrow[d] \arrow[r, hook] & \hilbtwo[Y] \\
    \fano(Y)
  \end{tikzcd}
\end{equation}
obtained from the universal family (or Fano correspondence)
\begin{equation}
  \begin{tikzcd}
    P \arrow[d] \arrow[r] & Y \\
    \fano(Y)
  \end{tikzcd}
\end{equation}
by taking the relative Hilbert square of~$P\to Y$ can be made into a standard flip diagram \eqref{equation:flip} where~$X$ is the Hilbert square of~$Y$ and~$F$ the Fano variety~$\fano(Y)$ of lines on~$Y$. The role of~$X'$ is played by a certain~$\mathbb{P}^n$\dash bundle over~$Y$. For more on the terminology and notation, one is referred to \cref{section:quadratic-fano-correspondence}. An excellent reference for more background on cubic hypersurfaces is \cite{huybrechts}.

This brings us to the second main result.
\begin{alphatheorem}
  \label{thm:cubic}
  Let~$Y\subset \PP^{n+1}$ be a smooth cubic hypersurface. Let~$\fano(Y)$ be its Fano variety of lines and~$\hilbtwo[Y]$ be its Hilbert square. Then there is a semiorthogonal decomposition
  \begin{equation}
    \label{equation:sod-quadratic-fano}
    \derived^\bounded(\hilbtwo[Y])=\langle \derived^\bounded(\fano(Y)), \underbrace{\derived^\bounded(Y),\ldots,\derived^\bounded(Y)}_{\text{$n+1$ copies}}\rangle.
  \end{equation}
  For the precise form of the functors one is referred to \cref{section:quadratic-fano-correspondence}.
\end{alphatheorem}

This is a derived categorical version, of a comparison between the Fano variety of lines and the Hilbert square of a cubic hypersurface, which was known to hold in various contexts, such as cohomology, classes in the Grothendieck ring of varieties, or (rational) Chow motives. In \cref{section:quadratic-fano-correspondence} we discuss these.

\subsection*{Fano visitors}
Using \cref{thm:cubic} we can study the Fano visitor problem for the Fano variety of lines on a cubic hypersurface. Recall that a smooth projective variety~$X$ is said to be a \emph{Fano visitor} if there exists a smooth projective Fano variety~$Y$ and a fully faithful functor~$\derived^\bounded(X)\hookrightarrow\derived^\bounded(Y)$. In this case we call~$Y$ a \emph{Fano host} for~$X$.

Bondal raised the question whether every smooth projective variety is a Fano visitor. A positive answer would imply that (additive invariants of) derived categories of Fano varieties are as complicated as those of arbitrary varieties. This question has been addressed in \cite{MR3628227,MR3490767,MR3954042,MR3713871,MR3954308,MR3764066} for various families of varieties, in particular it is known that curves, Enriques surfaces and complete intersections are Fano visitors.

To study the Fano visitor problem for Fano varieties of lines\footnote{To avoid any confusion with other usages of the name Fano, we will always write ``Fano variety of lines'' in full when we refer to~$\fano(Y)$, where~$Y$ is a cubic hypersurface.}, we prove the following result.
\begin{alphatheorem}
  \label{theorem:fano}
  Let~$n\geq 3$. Let~$Y$ be~$\mathbb{P}^n$, or a smooth quadric (resp.~cubic) hypersurface in~$\mathbb{P}^{n+1}$. Then~$\hilbtwo[Y]$ is Fano.
\end{alphatheorem}
This is in contrast to the case of smooth projective surface~$S$, in which case the Hilbert scheme~$\hilbn{m}{S}$ of~$m$ points is never a Fano variety.
We immediately obtain the following corollary to \cref{thm:cubic} and \cref{theorem:fano}.
\begin{alphacorollary}
  \label{corollary:fano-visitor}
  Let~$Y$ be a smooth cubic hypersurface of dimension~$n=3,4$. Then the Fano variety of lines~$\fano(Y)$ is a Fano visitor with Fano host~$\hilbtwo[Y]$.
\end{alphacorollary}
This is the first construction (to our knowledge) of a Fano host for a surface of general type which is not a complete intersection or a product of curves (for~$n=3$), and the first construction of a Fano host for a (higher-dimensional) hyperk\"ahler variety (for~$n=4$). For~$n\geq 5$ the Fano variety of lines is itself a Fano variety, so the Fano visitor problem is trivial.

\paragraph{Conventions}\quad
Throughout we will assume that~$\groundfield$ is an algebraically closed field of characteristic~0.

\paragraph{Acknowledgements}\quad
We want to thank Sergey Galkin, Daniel Huybrechts, Qingyuan Jiang, Robert Laterveer, Renjie Lyu, Will Sawin and Mingmin Shen for helpful discussions.

The results in this article were conceived during the ``Problems and recent developments in hyperk\"ahler geometry'' research seminar, organised by Daniel Huybrechts, and we heartily thank him. 

The first author was partially supported by a postdoctoral fellowship from the Research Foundation---Flanders (FWO).
The second author is supported by the Agence Nationale de la Recherche (ANR) under projects ANR-20-CE40-0023 and ANR-16-CE40-0011. He was also supported the Radboud Excellence Initiative program.
The third author is supported by a postdoctoral fellowship from the Research Foundation---Flanders (FWO).

\section{A fully faithfulness criterion for EZ-type functors}
Let~$F, Z$ and~$X$ be smooth projective varieties. Assume there is a smooth proper morphism~$\pi\colon Z\to F$ and a closed immersion~$i\colon Z\to X$ as in the following diagram.
\begin{equation}
  \label{diag:EZ}
  \begin{tikzcd}
    Z \arrow[r, hook, "i"] \arrow[d, "\pi"] & X \\
    F
  \end{tikzcd}
\end{equation}
Let~$c$ be the codimension of~$Z$ in~$X$. The following result is inspired by \cite[theorem~2.1]{MR2126495}, in the formulation of \cite[theorem~1.3]{MR3918435}. The terminology \emph{EZ-type} is taken from \cite{MR2126495}, where the diagram \eqref{diag:EZ} is denoted with~$E$ instead of~$F$.

\begin{proposition}
  \label{prop:FF}
  Assume that for every fiber~$P$ of~$\pi$ and all integers~$p, q$ with~$p+q>0$,
  \begin{equation}
    \label{eq:vanishing-cond}
    \HH^p(P,\bigwedge^{q}\mathcal{N}')=0,
  \end{equation}
  where~$\mathcal{N}'\coloneqq\normal_{Z/X}|_{P}$ is the restriction of the normal bundle of~$Z$ in~$X$ to~$P$.
  Then for every line bundle~$\mathcal{L}$ on~$Z$, the functor
  \begin{equation}
    \Phi\colon\derived^\bounded(F)\to\derived^\bounded(X):\mathcal{E} \mapsto i_*(\pi^*(\mathcal{E})\otimes\mathcal{L}),
  \end{equation}
  is fully faithful.
\end{proposition}

This is an application of the Bondal--Orlov criterion \cite[theorem~1.1]{alg-geom/9506012}, which we will quickly recall.

\begin{proposition}[(Bondal--Orlov criterion)]
  \label{proposition:bondal-orlov-criterion}
  Let~$X$ and~$Y$ be smooth projective varieties, and~$\mathcal{E}$ an object in~$\derived^\bounded(X \times Y)$. The corresponding Fourier--Mukai functor~$\Phi_{\mathcal{E}}\colon\derived^\bounded(X) \to \derived^\bounded(Y)$ is fully faithful if and only if for all closed points~$x,x'\in X$ we have that
  \begin{enumerate}[label=(\roman*)]
    \item\label{enumerate:bondal-orlov-1} $\Hom_Y(\Phi_{\mathcal{E}}(\mathcal{O}_x),\Phi_{\mathcal{E}}(\mathcal{O}_x))\cong\groundfield$;
    \item\label{enumerate:bondal-orlov-2} $\Hom_Y(\Phi_{\mathcal{E}}(\mathcal{O}_x),\Phi_{\mathcal{E}}(\mathcal{O}_x)[m])\cong0$ for all~$m\notin[0,\dim X]$;
    \item\label{enumerate:bondal-orlov-3} $\Hom_Y(\Phi_{\mathcal{E}}(\mathcal{O}_x),\Phi_{\mathcal{E}}(\mathcal{O}_{x'})[m])\cong0$ for all~$m\in\mathbb{Z}$ and~$x\neq x'$;
  \end{enumerate}
  where~$\mathcal{O}_x$ denotes the skyscraper sheaf at~$x$.
\end{proposition}

\begin{proof}[of \cref{prop:FF}]
  For a (closed) point~$x$ in~$F$, denote by~$P\coloneqq\pi^{-1}(x)$ the fiber. We denote~$j\colon P\hookrightarrow X$ for the closed immersion obtained by composition with~$i$. Then, $\Phi(\mathcal{O}_{x})=i_*(\mathcal{O}_P\otimes\mathcal{L})=j_*(\mathcal{L}|_P)$.

  Let us first check \ref{enumerate:bondal-orlov-3} of \cref{proposition:bondal-orlov-criterion}. Let~$P'$ denote the fiber~$\pi^{-1}(x')$ for~$x\neq x'$ in~$F$. Since~$P\cap P'=\emptyset$, the objects~$j_*(\mathcal{L}|_{P})$ and~$j_*(\mathcal{L}|_{P'})$ have disjoint support, hence they are completely orthogonal.

  To check~\ref{enumerate:bondal-orlov-1} and~\ref{enumerate:bondal-orlov-2}, we have by adjunction that
  \begin{equation}
    \label{equation:adjunction}
    \Hom_X(j_*(\mathcal{L}|_P),j_*(\mathcal{L}|_P)[m])
    \cong
    \Hom_P(\mathcal{L}|_P,j^!\circ j_*(\mathcal{L}|_P)[m]).
  \end{equation}
  By \cite[corollary~11.2 and proposition~11.8]{MR2244106}, the cohomology sheaves of the complex~$\LLL j^*\circ j_*(\mathcal{L}|_P)$ are given by
  \begin{equation}
    \mathcal{H}^{-s}(\LLL j^*\circ j_*(\mathcal{L}|_{P}))\cong \mathcal{L}|_P\otimes\bigwedge^s\mathcal{N}^\vee \text{ for } s=0, \dots, c+\dim F,
  \end{equation}
  where~$\mathcal{N}\coloneqq\normal_{P/X}$ is the normal bundle of~$P$ in~$X$. As~$j^!$ and~$\LLL j^*$ are related by Grothendieck duality, we have that
  \begin{equation}
    j^!\circ j_*(\mathcal{L}|_{P})=\LLL j^*\circ j_*(\mathcal{L}|_{P})\otimes\omega_j[\dim P-\dim X]=\LLL j^*\circ j_*(\mathcal{L}|_{P})\otimes \det(\mathcal{N})[-\dim F-c]
  \end{equation}
  and this object has cohomology sheaves
  \begin{equation}
    \mathcal{H}^{s}(j^!\circ j_*(\mathcal{L}|_{P}))\cong \mathcal{L}|_P\otimes\bigwedge^s\mathcal{N} \text{ for } s=0, \dots, c+\dim F.
  \end{equation}
  We can compute the right-hand side of \eqref{equation:adjunction} via the hypercohomology spectral sequence and we obtain the spectral sequence
  \begin{equation}
    \mathrm{E}_2^{p,q}
    =\HH^p(P,\mathcal{H}^{q}(j^{!}\circ j_*(\mathcal{L}|_P)\otimes\mathcal{L}^{-1}|_{P}))
    =\HH^p(P,\bigwedge^q\mathcal{N})
    \Rightarrow
    \Hom_X(j_*(\mathcal{L}|_{P}),j_*(\mathcal{L}|_{P})[p+q]).
  \end{equation}
  Hence~\ref{enumerate:bondal-orlov-1}, and~\ref{enumerate:bondal-orlov-2} for~$m<0$ follow immediately.

  To check~\ref{enumerate:bondal-orlov-2} for~$m>\dim F$, consider the following identification of the short exact sequence of normal bundles
  \begin{equation}
    \label{eq:nbseq}
    \begin{tikzcd}
      0 \arrow[r] & \normal_{P/Z} \arrow[r] \arrow[d, equals] & \normal_{P/X} \arrow[r] \arrow[d, equals] & \normal_{Z/X}|_P \arrow[r] \arrow[d, equals] & 0 \\
      0 \arrow[r] & \mathcal{O}_P^{\oplus\dim F} \arrow[r] & \mathcal{N} \arrow[r] & \mathcal{N}' \arrow[r] & 0
    \end{tikzcd}
  \end{equation}
  as the first term is isomorphic to~$\mathcal{O}_P\otimes_\groundfield\tangent_xF$. Therefore, for any positive integer~$q$, the bundle~$\bigwedge^q\mathcal{N}$ is a successive extension of direct sums of exterior powers of~$\mathcal{N}'$. From the vanishing hypothesis, we see that~$\HH^p(P,\bigwedge^q\mathcal{N})=0$ for all~$p+q>\dim F$. As a result,~$\mathrm{E}_\infty^{p, q}=0$ for all~$p+q>\dim F$, and we are done.
\end{proof}

\subsection{Some applications}
We will briefly discuss a few situations where this criterion applies, before proceeding with the setting of standard flips in \cref{section:standard-flip}.

\paragraph{Orlov's blowup formula}\quad
The easiest setting where one can apply \cref{prop:FF} is that of the usual blowup of a smooth projective variety~$X$ in a smooth center~$Z$ of codimension~$c\geq 2$, i.e.~we consider the diagram
\begin{equation}
  \label{eq:blowup}
  \begin{tikzcd}
    E \ar[hook]{r}{i} \ar{d}{\pi} & \Bl_Z X \\
    Z &
  \end{tikzcd}
\end{equation}
where~$i$ denotes the inclusion of the exceptional divisor~$E$. Then the vanishing condition \eqref{eq:vanishing-cond} is clearly satisfied as we have the identification~$\mathcal{N}'\cong\mathcal{O}_P(-1)$ using the notation of \cref{prop:FF}, and we see that the functor
\begin{equation}
  \Phi_m\colon\derived^\bounded(Z)\to\derived^\bounded(\Bl_Z X):\mathcal{E}\mapsto i_{*}(\pi^{*}(\mathcal{E})\otimes \mathcal{O}_E(mE))
\end{equation}
is fully faithful for every~$m\in\mathbb{Z}$, recovering \cite[assertion~4.2(i)]{MR1208153}, which is one of the ingredients for Orlov's blowup formula.

\paragraph{Krug--Ploog--Sosna's cyclic quotient singularities}\quad
More interestingly, \cref{prop:FF} also applies to the ``singular'' blowups of \cite{MR3811590}, for cyclic quotient singularities. The setting, using the notation of op.~cit., is as follows. Let~$Y$ denote a smooth quasiprojective variety with an action of a cyclic group~$G \cong \mu_m$. Then the fixed point locus~$S \subset Y$ (of codimension~$n$) is smooth and the quotient~$Y/G$ has rational singularities. We further assume that only the isotropy groups~$1$ and~$G$ occur, and that a generator of~$G$ acts on the normal bundle~$\normal_{S/Y}$ by multiplication with a fixed primitive~$m$th root of unity.

In this case there is a diagram
\begin{equation}
  \begin{tikzcd}
    Z\coloneqq\PP(\normal_{S/Y}) \ar[hook]{r}{i} \ar{d}{\pi} & X\coloneqq\Bl_S(Y/G)\cong(\Bl_SY)/G \\
    S
  \end{tikzcd}
\end{equation}
of smooth projective varieties. Here~$Y/G$ has cyclic quotient singularities, and the ``singular blowup''~$\Bl_S(Y/G)$ is a resolution of singularities.

By \cite[lemma~4.10(ii) and (vi)]{MR3811590}, we have that
\begin{equation}
  \normal_{Z/X}|_{P} \cong \mathcal{O}_{\PP^{n-1}}(-m)
\end{equation}
for every fiber~$P$ of~$\pi$, so \cref{prop:FF} applies as soon as~$n>m$ and one obtains fully faithful functors
\begin{equation}
  \Theta_\beta\colon\derived^\bounded(S) \to \derived^\bounded(\Bl_S(Y/G)): \mathcal{E} \mapsto i_{*}(\pi^{*}(\mathcal{E})\otimes \mathcal{O}_{\pi}(\beta)),
\end{equation}
for~$\beta \in \mathbb{Z}$, as in \cite[theorem~4.1(ii)]{MR3811590}, which is one of the ingredients for the Krug--Ploog--Sosna semiorthogonal decomposition.

A particularly relevant example for this paper is given by taking~$Y\coloneqq S\times S$ with~$S$ a smooth projective variety, and~$G=\mu_2$ acting via transposition; in this case~$X \cong \hilbtwo[S]$. The following result is a specific case of \cite[theorem~4.1(ii)]{MR3811590} with $m=2$.
\begin{proposition}[(Krug--Ploog--Sosna)]
  Let~$S$ be a smooth projective variety of dimension~$n\geq 2$. Then there exists a semiorthogonal decomposition
  \begin{equation}
    \derived^\bounded(\hilbtwo[S])
    =
    \langle
    \derived^\bounded_{\mu_2}(S^2),
    \underbrace{
      \derived^\bounded(S),
      \ldots,
      \derived^\bounded(S)
    }_{\text{$n-2$ copies}}
    \rangle.
  \end{equation}
\end{proposition}

\paragraph{Cayley's trick}\quad
A third application is Cayley's trick, which gives a relationship between a complete intersection and a canonically associated hypersurface in a projective bundle. This relationship can be studied at different levels, and the name ``Cayley's trick'' originates from \cite{MR3299729} where it was used to study cohomology. In \cite{MR3628227} Cayley's trick was used on the level of derived categories to study Fano visitors as in \cref{subsection:fano-visitor}, and we will now explain how to obtain part of their semiorthogonal decomposition.

Let~$\mathcal{E}$ denote a vector bundle of rank~$r$ on a smooth variety~$Y$, and~$s \in \HH^0(Y,\mathcal{E})$ a regular section with smooth projective zero locus~$\ZZ(s)$. By the isomorphism
\begin{equation}
  \HH^0(Y,\mathcal{E}) \cong \HH^0(\PP(\mathcal{E}^{\vee}),\mathcal{O}_{\PP(\mathcal{E}^{\vee})}(1)),
\end{equation}
the section~$s$ also defines a section~$f_s \in \HH^0(\PP(\mathcal{E}^{\vee}),\mathcal{O}_{\PP(\mathcal{E}^{\vee})}(1))$ and we denote~$\ZZ(f_s)$ its zero locus, which defines a hypersurface in~$\PP(\mathcal{E}^{\vee})$. In particular, there is a cartesian diagram
\begin{equation}
  \begin{tikzcd}
    Z\coloneqq\PP(\normal_{\ZZ(s)/Y}) \arrow[r, hook, "i"] \arrow[d, "\pi"] \arrow[dr, phantom, "\square"] & X\coloneqq\ZZ(f_s) \ar{d} \\
    \ZZ(s) \ar[hook]{r} & Y
  \end{tikzcd}
\end{equation}
where~$X\to Y$ is a~$\PP^{r-2}$-bundle over~$Y\backslash\ZZ(s)$ and the restriction~$\pi$ is a~$\PP^{r-1}$-bundle.
In this case, the normal bundle is
\begin{equation}
  \normal_{Z/X}|_{P} \cong \Omega_{\PP^{r-1}}(1),
\end{equation}
for a fiber~$P \cong \PP^{r-1}$ of~$\pi$, and \cref{prop:FF} applies by Bott vanishing. This yields fully faithful functors
\begin{equation}
  \Phi\colon\derived^\bounded(\ZZ(s))\to\derived^\bounded(\ZZ(f_s)):\mathcal{E}\mapsto i_{*}(\pi^{*}(\mathcal{E})\otimes \mathcal{L}),
\end{equation}
for every line bundle~$\mathcal{L}$ on~$\PP(\normal_{\ZZ(s)/Y})$, recovering part of \cite[proposition 2.10]{MR2437083}.

\section{A semiorthogonal decomposition for standard flips}
\label{section:standard-flip}
In this section we prove \cref{theorem:flip}. 

The proof of part~\ref{enumerate:part-2} of \cref{theorem:flip} is more interesting, and is very much inspired by Kuznetsov's homological projective duality \cite{MR2354207} and its interpretation by Thomas \cite{MR3821163} via the ``chess game'' from \cite{1704.01050v1} (but the setting in op.~cit.~is different from what we consider here\footnote{The authors of \cite{1704.01050v1} have communicated to us that they were also aware of a proof of \cref{theorem:flip} using chess game methods similar to ours.}.)

In the case~$\ell=1$, a slightly different (and shorter) proof of a stronger version of \cref{theorem:flip}\ref{enumerate:part-2} is given in appendix~\ref{app:l1}, using only usual mutation techniques (and not ``chess game methods'' to streamline the more complicated mutations required for the result).

The semiorthogonal decomposition \eqref{eqn:goal} will be established in several steps:
\begin{enumerate}
  \item The fully faithfulness of the functor~$\Phi_{m}$ on~$\derived^\bounded(F)$ follows from \cref{prop:FF}.
  \item The fully faithfulness of the functor~$\RRR\tau_{*}\circ\LLL\tau'^{*}$ on~$\derived^\bounded(X')$ is Bondal--Orlov's theorem, see \cref{thm:BO} and also \cref{rmk:ProofOfBO}.
  \item The semiorthogonality of the subcategories in \eqref{eqn:goal} is the combination of \cref{SO-FF} and \cref{SO-X'F}.
  \item The fullness of the semiorthogonal decomposition \eqref{eqn:goal} is \cref{cor:generation}.
\end{enumerate}

Let us first introduce some notation. Assume we are given a standard flip diagram~\eqref{equation:flip}. For integers~$a,b\in\mathbb{Z}$, denote by
\begin{equation}
  \mathcal{O}(a, b)\coloneqq p^{*}\mathcal{O}_{\pi}(a)\otimes p^{\prime*}\mathcal{O}_{\pi'}(b),
\end{equation}
which is a line bundle on~$E$. We then define the following triangulated subcategory of~$\derived^\bounded(\tilde X)$:
\begin{equation}
  \label{equation:subcategory}
  \mathcal{A}(a, b)\coloneqq j_{*}(p^{*}\circ \pi^{*}\derived^\bounded(F)\otimes \mathcal{O}(a, b)).
\end{equation}
This is a triangulated subcategory because for all~$a,b\in\mathbb{Z}$
the functor is fully faithful by the projective bundle formula and the blowup formula,
and the essential image of a full functor is naturally a triangulated subcategory.

We note first that the components
\begin{equation}
  \Phi_m(\derived^\bounded(F))\coloneqq i_{*}(\pi^{*}(\derived^\bounded(F))\otimes \mathcal{O}_{\pi}(m)),
\end{equation}
for~$m\in\mathbb{Z}$ appearing in the decomposition~\eqref{eqn:goal} can be expressed in terms of the categories just defined, by the following lemma.

\begin{lemma}
  \label{lemma:PhiA}
  For every~$m\in\mathbb{Z}$ we have the identification of subcategories
  \begin{equation}
    \Phi_{m}(\derived^\bounded(F))=\RRR\tau_{*}\mathcal{A}(m, 0)
  \end{equation}
  in~$\derived^\bounded(X)$.
\end{lemma}

\begin{proof}
	Notation is as in diagram~\eqref{equation:flip}.
  This follows from the computation
  \begin{equation}
    \begin{aligned}
      \RRR\tau_{*}\mathcal{A}(m, 0)
      &=\RRR\tau_{*}\circ j_{*}(p^{*}\circ\pi^{*}\derived^\bounded(F)\otimes p^{*}\mathcal{O}_{\pi}(m))\\
      &=i_{*}\circ\RRR p_{*}\circ p^{*}(\pi^{*}\derived^\bounded(F)\otimes \mathcal{O}_{\pi}(m))\\
      &=i_{*}(\pi^{*}\derived^\bounded(F)\otimes \mathcal{O}_{\pi}(m))\\
      &=\Phi_{m}(\derived^\bounded(F)).
    \end{aligned}
  \end{equation}
\end{proof}

For convenience, we also define the following subcategories of~$\derived^\bounded(\tilde X)$, which are equivalent to~$\derived^\bounded(Z')$ and~$\derived^\bounded(Z)$ respectively by Orlov's blowup formula.
\begin{equation}
  \begin{aligned}
    \mathcal{A}(a,\star)
    &\coloneqq j_{*}(p^{*}\mathcal{O}_{\pi}(a)\otimes p^{\prime*}\derived^\bounded(Z')), \\
    \mathcal{A}(\star,b)
    &\coloneqq j_{*}(p^{*}\derived^\bounded(Z)\otimes p^{\prime*}\mathcal{O}_{\pi'}(b)).
  \end{aligned}
\end{equation}

Since $Z$ and $Z'$ are projective bundles over $F$, by Orlov's projective bundle formula \cite[theorem~2.6]{MR1208153}, there are the following semiorthogonal decompositions for every~$m\in\mathbb{Z}$:
\begin{equation}
  \label{eqn:Astar1}
  \mathcal{A}(a,\star)=\langle\mathcal{A}(a, m-l), \mathcal{A}(a, m-l+1), \dots, \mathcal{A}(a, m)\rangle.
\end{equation}
Similarly we have the decompositions
\begin{equation}
  \label{eqn:Astar2}
  \mathcal{A}(\star,b)=\langle\mathcal{A}(m-k, b), \mathcal{A}(m-k+1, b), \dots, \mathcal{A}(m, b)\rangle.
\end{equation}

Applying Orlov's blowup formula \cite[theorem~4.3]{MR1208153} to the blowup~$\tau\colon\tilde X\to X$, we have the semiorthogonal decomposition:
\begin{equation}
  \label{eqn:Blup1}
  \derived^\bounded(\tilde X)=\langle \mathcal{A}(\star, -l), \mathcal{A}(\star, -l+1), \dots, \mathcal{A}(\star, -1), \LLL\tau^{*}\derived^\bounded(X)\rangle.
\end{equation}
Similarly, applying the blowup formula to $\tau'\colon\tilde X\to X$, we have
\begin{equation}\label{eqn:Blup2}
  \derived^\bounded(\tilde X)=\langle \mathcal{A}(-k, \star), \mathcal{A}(-k+1, \star), \dots, \mathcal{A}(-1, \star), \LLL\tau^{\prime*}\derived^\bounded(X')\rangle.
\end{equation}
Inserting appropriately chosen \eqref{eqn:Astar1} and \eqref{eqn:Astar2} into \eqref{eqn:Blup1} and \eqref{eqn:Blup2} we get different semiorthogonal decompositions of $\derived^\bounded(\tilde X)$. The main theme of the proof is to compare them. The following vanishing result plays a central role in the argument, and is similar to \cite[lemma~4.3]{MR3821163}, but we are not in the setting of homological projective duality.

\begin{lemma}
  \label{vanishing}
  Let~$a_{1}, a_{2}, b_{1}, b_{2}$ be integers. Then~$\RHom_{\tilde X}(\mathcal{A}(a_{1}, b_{1}), \mathcal{A}(a_{2}, b_{2}))=0$ in any of the following cases:
  \begin{enumerate}
    \item $1\leq a_{1}-a_{2}\leq k-1$,
    \item $1\leq b_{1}-b_{2}\leq \ell-1$,
    \item $a_{1}-a_{2}=k$ and $0\leq b_{1}-b_{2}\leq \ell-1$,
    \item $b_{1}-b_{2}=\ell$ and $0\leq a_{1}-a_{2}\leq k-1$.
  \end{enumerate}
\end{lemma}

\begin{proof}
  For arbitrary~$\mathcal{F}_{1}, \mathcal{F}_{2}\in \derived^\bounded(F)$, by adjunction,
  \begin{equation}
    \label{eqn:HomAA}
    \begin{aligned}
      &\RHom_{\tilde X}\left( j_{*}(p^{*}\circ\pi^{*}\mathcal{F}_{1}\otimes \mathcal{O}(a_{1}, b_{1})), j_{*}(p^{*}\circ \pi^{*}\mathcal{F}_{2}\otimes \mathcal{O}(a_{2}, b_{2})) \right) \\
      &\quad=\RHom_E\left( \LLL j^{*}\circ j_{*}(p^{*}\circ\pi^{*}\mathcal{F}_{1}\otimes \mathcal{O}(a_{1}, b_{1})), p^{*}\circ \pi^{*}\mathcal{F}_{2}\otimes \mathcal{O}(a_{2}, b_{2}) \right)
    \end{aligned}
  \end{equation}
  By \cite[corollary~11.4(ii)]{MR2244106}, there is a distinguished triangle:
  \begin{equation}
  \label{eq:jtriangle}
    A\otimes \mathcal{O}_{E}(-E)[1]
    \to \LLL j^{*}\circ j_{*}A
    \to A
    \xrightarrow{+1}
  \end{equation}
  which we consider for~$A=p^*\circ\pi^*\mathcal{F}_1\otimes\mathcal{O}(a_1,b_1)$.
  Since $\mathcal{O}_{E}(E)=\mathcal{O}(-1,-1)$, this triangle reduces to
  \begin{equation}
    p^{*}\circ\pi^{*}\mathcal{F}_{1}\otimes \mathcal{O}(a_{1}+1, b_{1}+1)[1]
    \to\LLL j^{*}\circ j_{*}(p^{*}\pi^{*}\mathcal{F}_{1}\otimes \mathcal{O}(a_{1}, b_{1}))
    \to p^{*}\circ\pi^{*}\mathcal{F}_{1}\otimes \mathcal{O}(a_{1}, b_{1})
    \xrightarrow{+1}
  \end{equation}
  Applying the functor~$\RHom_E(-, p^{*}\circ\pi^{*}\mathcal{F}_{2}\otimes \mathcal{O}(a_{2}, b_{2}))$, 
  we see that (the right-hand side of) \eqref{eqn:HomAA} is the cone of the morphism
 \begin{equation}
 	\begin{aligned}
 	 \RHom_E(p^{*}\circ\pi^{*}\mathcal{F}_{1}\otimes \mathcal{O}(a_{1}+1, b_{1}+1)[2], p^{*}\circ\pi^{*}\mathcal{F}_{2}\otimes \mathcal{O}(a_{2}, b_{2}))\\
 	\xrightarrow{f}\RHom_E(p^{*}\circ\pi^{*}\mathcal{F}_{1}\otimes \mathcal{O}(a_{1}, b_{1}), p^{*}\circ\pi^{*}\mathcal{F}_{2}\otimes \mathcal{O}(a_{2}, b_{2}))
 	\end{aligned}
 \end{equation}

  or equivalently, denoting $a':=a_2-a_1$, $b':=b_2-b_1$, \eqref{eqn:HomAA} is the cone of
  \begin{equation}
    \label{eqn:HomFF}
    \RHom_E(p^{*}\circ\pi^{*}\mathcal{F}_{1}[2], p^{*}\circ\pi^{*}\mathcal{F}_{2}\otimes \mathcal{O}(a'-1, b'-1))
    \xrightarrow{f}\RHom_E(p^{*}\circ\pi^{*}\mathcal{F}_{1}, p^{*}\pi^{*}\mathcal{F}_{2}\otimes \mathcal{O}(a', b')).
  \end{equation}
  Now note that for all~$a,b\in\mathbb{Z}$,
  \begin{equation}
    \begin{aligned}
      &\RHom_E(p^{*}\circ\pi^{*}\mathcal{F}_{1}, p^{*}\circ\pi^{*}\mathcal{F}_{2}\otimes \mathcal{O}(a, b)) \\
      &\quad=\RHom_E(p^{*}\circ\pi^{*}\mathcal{F}_{1}, p^{*}\circ\pi^{*}\mathcal{F}_{2}\otimes p^{*}\mathcal{O}_{\pi}(a)\otimes p^{\prime*}\mathcal{O}_{\pi'}(b))\\
      &\quad=\RHom_Z(\pi^{*}\mathcal{F}_{1}, \pi^{*}\mathcal{F}_{2}\otimes\mathcal{O}_{\pi}(a)\otimes \RRR p_{*}\circ p^{\prime*}\mathcal{O}_{\pi'}(b))\\
      &\quad=\RHom_Z(\pi^{*}\mathcal{F}_{1}, \pi^{*}\mathcal{F}_{2}\otimes\mathcal{O}_{\pi}(a)\otimes \pi^{*}\circ\RRR\pi'_{*}\mathcal{O}_{\pi'}(b))\\
      &\quad=\RHom_F(\mathcal{F}_{1}, \mathcal{F}_{2}\otimes\RRR\pi_{*}\mathcal{O}_{\pi}(a)\otimes \RRR\pi'_{*}\mathcal{O}_{\pi'}(b)),
    \end{aligned}
  \end{equation}
  which vanishes when~$a\in [-k, -1]$ or~$b\in [-\ell, -1]$. Therefore, in each of the cases in the statement, both the source and the target of~$f$ in \eqref{eqn:HomFF} vanish, hence also the cone.
\end{proof}

\subsection{Semiorthogonality}
We will first prove that the subcategories in \eqref{eqn:goal} are semiorthogonal. As e.g.~in the proof of Orlov's blowup formula, this falls apart into two statements:
\begin{enumerate}[label=(\roman*)]
  \item\label{enumerate:semiorthogonality-1} the semiorthogonality of the subcategories~$\Phi_{-k+\ell}(\derived^\bounded(F)),\dots, \Phi_{-1}(\derived^\bounded(F))$;
  \item\label{enumerate:semiorthogonality-2} the semiorthogonality of the pair~$\Phi_{m}(\derived^\bounded(F)),\RRR\tau_*\circ\LLL\tau^{\prime*}\derived^\bounded(X')$ for~$m=-k+\ell,\ldots,-1$.
\end{enumerate}

Let us first prove~\ref{enumerate:semiorthogonality-1}.
\begin{proposition}
  \label{SO-FF}
  If~$m'< m$ are two integers such that~$m-m'<k-\ell$, then for all~$\mathcal{E}, \mathcal{F}\in \derived^\bounded(F)$, we have
  \begin{equation}
    \RHom_X(\Phi_{m}(\mathcal{E}), \Phi_{m'}(\mathcal{F}))=0.
  \end{equation}
\end{proposition}

\begin{proof}
  It suffices to prove this for~$\mathcal{E}$ and~$\mathcal{F}$ coherent sheaves, by the standard d\'evissage argument for bounded complexes of coherent sheaves. By adjunction,
  \begin{equation}
    \begin{aligned}
      \RHom_X(\Phi_{m}(\mathcal{E}), \Phi_{m'}(\mathcal{F}))
      &\cong\RHom_X\left(i_{*}(\pi^{*}\mathcal{E}\otimes \mathcal{O}_{\pi}(m)), i_{*}(\pi^{*}\mathcal{F}\otimes \mathcal{O}_{\pi}(m'))\right) \\
      &\cong\RHom_Z\left(\pi^{*}\mathcal{E}\otimes \mathcal{O}_{\pi}(m), i^{!}\circ i_{*}(\pi^{*}\mathcal{F}\otimes \mathcal{O}_{\pi}(m'))\right).
    \end{aligned}
  \end{equation}
  To show this space vanishes, it suffices to show that
  \begin{equation}
    \label{eq:homcoh}
    \RHom_Z\left(\pi^{*}\mathcal{E}\otimes \mathcal{O}_{\pi}(m), \mathcal{H}^q(i^{!}\circ i_{*}(\pi^{*}\mathcal{F}\otimes \mathcal{O}_{\pi}(m'))[\ast]\right)=0,
  \end{equation}
  for all $q \in \mathbb{Z}$. By \cite[lemma~1.4]{MR3918435}, the Fourier--Mukai kernel of~$i^{!}\circ i_{*}\colon\derived^\bounded(Z) \to \derived^\bounded(Z)$ has cohomology sheaves~$\mathcal{H}^{q}=\Delta_{Z, *}(\bigwedge^{q}\mathcal{N})$, where~$\mathcal{N}\coloneqq\normal_{Z/X}\cong\mathcal{O}_{\pi}(-1)\otimes \pi^{*}(\mathcal{V}')$ for some vector bundle~$\mathcal{V}'$ on~$F$ of rank~$\ell+1$ by the standing hypothesis. We then find:
  \begin{equation}
  \begin{aligned}
    &\RHom_Z\left(\pi^{*}\mathcal{E}\otimes \mathcal{O}_{\pi}(m), \mathcal{H}^q(i^{!}\circ i_{*}(\pi^{*}\mathcal{F}\otimes \mathcal{O}_{\pi}(m'))[\ast]\right)\\
    &\cong\RHom_Z\left(\pi^{*}\mathcal{E}, \pi^{*}\mathcal{F}\otimes \mathcal{O}_{\pi}(m'-m)\otimes \bigwedge^{q}\mathcal{N}[\ast]\right) \\
    &\cong\RHom_Z\left(\pi^{*}\mathcal{E}, \pi^{*}\mathcal{F}\otimes \mathcal{O}_{\pi}(m'-m)\otimes\mathcal{O}_{\pi}(-q)\otimes \bigwedge^{q}\pi^{*}(\mathcal{V}')[\ast]\right) \\
    &\cong\RHom_Z\left(\pi^{*}\mathcal{E}, \pi^{*}(\mathcal{F}\otimes \bigwedge^{q}\mathcal{V}')\otimes \mathcal{O}_{\pi}(m'-m-q)[\ast]\right)
  \end{aligned}
  \end{equation}

  As~$\rk\mathcal{V}'=\ell+1$, this space is zero except possibly for~$0\leq q\leq \ell+1$. Combined with the assumption that~$m'-m\in [\ell-k+1, -1]$, we see that~$m'-m-q\in [-k, -1]$.

  Now~$\pi^{*}\mathcal{E}\in \pi^{*}\derived^\bounded(F)$ whilst~$\pi^{*}(\mathcal{F}\otimes \bigwedge^{q}\mathcal{V}')\otimes \mathcal{O}_{\pi}(m'-m-q)\in \pi^{*}\derived^\bounded(F)\otimes \mathcal{O}_{\pi}(m'-m-q)$, so using the semiorthogonality in Orlov's projective bundle formula, we deduce that~\eqref{eq:homcoh} vanishes, and the proposition follows.
\end{proof}

Let us show the remaining semiorthogonality \cref{enumerate:semiorthogonality-2}. We use an argument from the original proof of \cref{thm:BO} of Bondal--Orlov from \cite[theorem~3.6]{alg-geom/9506012}. For the convenience of the reader, we give a complete proof.

\begin{proposition}
  \label{SO-X'F}
  For all integers~$m\in [-k+\ell, -1]$, and for all~$\mathcal{F}\in \derived^\bounded(X')$ and~$\mathcal{G}\in \derived^\bounded(F)$, we have
  \begin{equation}
    \RHom_X\left(\RRR\tau_{*}\circ\LLL\tau^{\prime*}(\mathcal{F}),\Phi_{m}(\mathcal{G})\right)=0.
  \end{equation}
\end{proposition}

\begin{proof}
  By \cref{lemma:PhiA}, we are to show that~$\RHom_X(\RRR\tau_{*}\circ\LLL\tau^{\prime*}\derived^\bounded(X'), \RRR\tau_{*}\mathcal{A}(m, 0))=0$. By adjunction, it is enough to show that
  \begin{equation}
    \label{eqn:X'An0}
    \RHom_{\tilde X}(\LLL\tau^{*}\circ\RRR\tau_{*}\circ \LLL\tau^{\prime*}\derived^\bounded(X'),\mathcal{A}(m, 0))=0.
  \end{equation}
  Take any~$\mathcal{E}\in\LLL\tau^{\prime*}\derived^\bounded(X')$, viewed as an object in~$\derived^\bounded(\tilde X)$. Then~$\LLL\tau^{*}\circ\RRR\tau_{*}(\mathcal{E})$ is its right projection to the component~$\LLL\tau^{*}\derived^\bounded(X)$ with respect to the semiorthogonal decomposition~\eqref{eqn:Blup1}.

  Using \eqref{eqn:Astar2}, we will use the following version of \eqref{eqn:Blup1}:
  \begin{equation}
    \label{equation:before}
    \begin{aligned}
      \derived^\bounded(\tilde X)=\langle
      \mathcal{A}(-k, -\ell), \mathcal{A}(-k+1, -\ell), \dots, \mathcal{A}(-1, -\ell), \mathcal{A}(0, -\ell) ; \quad\\
       \mathcal{A}(-k+1, -\ell+1), \dots, \mathcal{A}(0, -\ell+1), \mathcal{A}(1, -\ell+1); \quad \\
      \dots;\quad \\
      \mathcal{A}(-k+\ell-1, -1), \mathcal{A}(-k+\ell, -1), \dots, \mathcal{A}(\ell-1, -1);\quad \\
      \LLL\tau^{*}\derived^\bounded(X)\rangle.
    \end{aligned}
  \end{equation}
  In \cref{figure:mutation-example}(\subref{figure:before}) we make the order of terms~$\mathcal{A}(i,j)$ explicit in the case of~$k=6$ and~$\ell=4$.

  \begin{figure}[t]
    \centering

    \caption{Visualising the mutation of the semiorthogonal decomposition in \eqref{equation:before}, for~$k=6$ and~$\ell=4$.}
    \label{figure:mutation-example}

    \begin{subfigure}[b]{\textwidth}
      \centering
      \begin{tikzpicture}[scale=0.8]
        \draw[thick, ->] (-7,0) -- (4,0) node[right] {$k$};
        \draw[thick, ->] (0,-5) -- (0,1) node[above] {$\ell$};

        \draw[very thin,color=gray] (-6.9,-4.9) grid (3.9,0.9);

        \foreach \l in {0,...,3} {
          \foreach \k in {1,...,7} {
            \pgfmathtruncatemacro\i{\k + \l*7}
            \draw[fill] (-7+\k+\l,-4+\l) circle (2pt) node [right] {\i};
          }
        }
      \end{tikzpicture}
      \caption{Before the mutation}
      \label{figure:before}
    \end{subfigure}

    \begin{subfigure}[b]{\textwidth}
      \centering
      \begin{tikzpicture}[scale=0.8]
        \draw [rounded corners, fill=black!20] (-7,-4.3) -- node [below, near start] {$\mathcal{D}_1$} (-0.7,-4.3) -- (-0.7,-0.7) -- (-3.5,-0.7) -- cycle;
        \draw [rounded corners, fill=black!20] (-0.2,-4.3) -- node [below, near end] {$\mathcal{D}_2$} (0.7,-4.3) -- (4.3,-0.7) -- (-0.2,-0.7) -- cycle;

        \draw[thick, ->] (-7,0) -- (4,0) node[right] {$k$};
        \draw[thick, ->] (0,-5) -- (0,1) node[above] {$\ell$};

        \draw[very thin, color=gray] (-6.9,-4.9) grid (3.9,0.9);

        \foreach \k/\l [count=\i] in {
          -6/-4, -5/-4, -4/-4, -3/-4, -2/-4, -1/-4,
          -5/-3, -4/-3, -3/-3, -2/-3, -1/-3,
          -4/-2, -3/-2, -2/-2, -1/-2,
          -3/-1, -2/-1, -1/-1 } {
          \draw[fill] (\k,\l) circle (2pt) node [left] {\i};
        }
        \foreach \k/\l [count=\i] in {
          0/-4,
          0/-3, 1/-3,
          0/-2, 1/-2, 2/-2,
          0/-1, 1/-1, 2/-1, 3/-1 } {
          \pgfmathtruncatemacro\j{\i + 18}
          \draw[fill] (\k,\l) circle (2pt) node [right] {\j};
        }
      \end{tikzpicture}
      \caption{After the mutation, with the grouping according to \eqref{equation:grouping}}
      \label{figure:after}
    \end{subfigure}
  \end{figure}

  The virtue of this decomposition is that, thanks to \cref{vanishing}, among all the components appearing above, there are no~$\Hom$'s from~$\mathcal{A}(a, b)$ to~$\mathcal{A}(a', b')$ with~$a\geq 0$ and~$a'<0$. Therefore, we can rearrange the order of the components to obtain the following semiorthogonal decomposition
  \begin{equation}
    \derived^\bounded(\tilde X)=\langle \mathcal{D}_{1}, \mathcal{D}_{2}, \LLL\tau^{*}\derived^\bounded(X)\rangle,
  \end{equation}
  where
  \begin{equation}
    \label{equation:grouping}
    \begin{aligned}
      \mathcal{D}_{1}
      &\coloneqq\langle\mathcal{A}(-k, -\ell),  \dots, \mathcal{A}(-1, -\ell);  \mathcal{A}(-k+1, -\ell+1), \dots, \mathcal{A}(-1, -\ell+1); \\
      &\qquad\dots; \mathcal{A}(-k+\ell-1, -1), \dots, \mathcal{A}(-1, -1) \rangle, \\
      \mathcal{D}_{2}
      &\coloneqq\langle  \mathcal{A}(0, -\ell) ;\mathcal{A}(0, -\ell+1), \mathcal{A}(1, -\ell+1); \dots; \mathcal{A}(0, -1), \dots, \mathcal{A}(\ell-1, -1)\rangle.
    \end{aligned}
  \end{equation}
  See \cref{figure:mutation-example}(\subref{figure:after}) for the order of the terms~$\mathcal{A}(i,j)$ after the mutation for the case~$k=6$ and~$\ell=4$.

  Since~$\mathcal{D}_{1}\subset \langle\mathcal{A}(-k, \star), \dots, \mathcal{A}(-1, \star)\rangle$, using the semiorthogonality of~\eqref{eqn:Blup2}, we see that~$\mathcal{E}\in {}^{\perp}\mathcal{D}_{1}=\langle \mathcal{D}_{2}, \LLL\tau^{*}\derived^\bounded(X)\rangle$.

  We deduce that there is an object~$\mathcal{E}'\in \mathcal{D}_{2}$ fitting into the following distinguished triangle
  \begin{equation}
    \LLL\tau^{*}\circ\RRR\tau_{*}\mathcal{E}\to \mathcal{E}\to \mathcal{E}'\xrightarrow{+1}.
  \end{equation}
  Therefore, to show that~$\RHom_{\tilde X}(\LLL\tau^{*}\circ\RRR\tau_{*}\mathcal{E}, \mathcal{A}(m, 0))=0$ for all~$m\in [-k+\ell, -1]$, it suffices to show that there is no morphism from~$\mathcal{E}$ or~$\mathcal{E}'$ to~$\mathcal{A}(m, 0)$.

  For the vanishing of~$\RHom_X(\mathcal{E}, \mathcal{A}(m,0))$, it suffices to apply the semiorthogonality from~\eqref{eqn:Blup2}.

  For the second vanishing, since~$\mathcal{E}'\in \mathcal{D}_{2}$, to see that~$\RHom_X(\mathcal{E}', \mathcal{A}(m,0))=0$, it is enough to check that
  \begin{equation}
    \RHom_X(\mathcal{A}(a, b), \mathcal{A}(m, 0))=0
  \end{equation}
  for any~$(a, b)$ with~$\ell-1\geq a\geq 0$ and~$-1\geq b\geq a-\ell$. But this holds by \cref{vanishing}, since $a-m\leq \ell-1-(-k+\ell)=k-1$.

  The vanishing~\eqref{eqn:X'An0} is proved and so is the proposition.
\end{proof}

\begin{remark}[(Bondal--Orlov's fully faithfulness)]\label{rmk:ProofOfBO}
  In the literature the proof of \cref{thm:BO} is given only in the case that~$F$ is a point (see~e.g.~\cite[proposition~11.23]{MR2244106}, or the original~\cite[theorem~3.6]{alg-geom/9506012}), but fully faithfulness of the functor~$\RRR\tau_{*}\circ\LLL\tau'^{*}$ in the general case follows with hardly any extra work. Indeed, by adjunction and the fully faithfulness of~$\LLL\tau'^{*}$, we are reduced to show that for all~$\mathcal{E}, \mathcal{F}\in \LLL\tau'^{*}\derived^\bounded(X')$, the natural map~$\Hom_{\tilde X}(\mathcal{E}, \mathcal{F})\to \Hom_{\tilde X}(\LLL\tau^{*}\circ\RRR\tau_{*}\mathcal{E}, \mathcal{F})$ is an isomorphism.

  Keeping the notation from the proof of \cref{SO-X'F}, it is enough to show that~$\Hom_{\tilde X}(\mathcal{E}', \mathcal{F})=0$. By adjunction, this amounts to the vanishing $\RRR\tau'_{*}\left(\mathcal{A}(a,b)\otimes \omega_{\tau'}^{\vee}\right)=0$ for all~$(a, b)$ with $\ell-1\geq a\geq 0$ and~$-1\geq b\geq a-\ell$, which follows again from \cref{vanishing}. We stress that this \emph{is} the original proof of Bondal--Orlov.
\end{remark}

\subsection{Generation}
We follow Thomas' presentation in \cite{MR3821163} of Kuznetsov's argument for ``mutation-cancellation'', which is initially in the context of homological projective duality \cite{MR2354207}. This will be used to show that the sequence of subcategories in \eqref{eqn:goal}, shown to be semiorthogonal in the previous section, moreover generate the category~$\derived^\bounded(X)$.
The main result of this section is the following.

\begin{proposition}
  \label{prop:Generation}
  For all~$m\in [-k, -k+\ell-1]$, the subcategory~$\mathcal{A}(m, 0)$ is in the triangulated subcategory of~$\derived^\bounded(\tilde X)$ generated by the following subcategories:
  \begin{equation}
    \label{eqn:span}
    \mathcal{A}(\star, -\ell), \dots, \mathcal{A}(\star, -1); \mathcal{A}(-k+\ell, 0), \dots, \mathcal{A}(-1, 0), \LLL\tau^{\prime*}\derived^\bounded(X').
  \end{equation}
\end{proposition}

Assuming this proposition, let us deduce the generation part in \cref{theorem:flip}\ref{enumerate:part-2}.

\begin{corollary}\label{cor:generation}
  The triangulated category~$\derived^\bounded(X)$ is generated by the subcategories
  \begin{equation}
    \Phi_{-k+\ell}(\derived^\bounded(F)),\dots, \Phi_{-1}(\derived^\bounded(F)), \RRR\tau_{*}\circ\LLL\tau^{\prime*}(\derived^\bounded(X'))
  \end{equation}
\end{corollary}

\begin{proof}
  Let~$\mathcal{C}\subset \derived^\bounded(\tilde X)$ be the triangulated subcategory generated by the subcategories in~\eqref{eqn:span}. By definition, $\mathcal{C}$ contains $\mathcal{A}(a,b)$ for all $a\in [-k, -1]$ and $b\in [-l, -1]$.

  Thanks to \cref{prop:Generation}, $\mathcal{C}$ also contains $\mathcal{A}(-k, 0), \dots, \mathcal{A}(-1, 0)$. Therefore, $\mathcal{C}$ contains~$\mathcal{A}(a, b)$ for all~$a\in [-k, -1]$ and~$b\in [-\ell, 0]$. By~\eqref{eqn:Astar1} (with~$m=0$), $\mathcal{C}$ contains~$\mathcal{A}(a,\star)$ for all~$a\in [-k, -1]$. By \eqref{eqn:Blup2}, we deduce that in fact
  \begin{equation}
    \mathcal{C}=\derived^\bounded(\tilde X),
  \end{equation}
  i.e.~$\derived^\bounded(\tilde X)$ is generated by the subcategories in~\eqref{eqn:span}. Applying the projection functor~$\RRR\tau_{*}$ which is essentially surjective, one finds that~$\derived^\bounded(X)$ is generated by the subcategories
  \begin{equation}
    \RRR\tau_{*}\mathcal{A}(\star, -l), \dots,\RRR \tau_{*}\mathcal{A}(\star, -1); \RRR\tau_{*}\mathcal{A}(-k+\ell, 0), \dots, \RRR\tau_{*}\mathcal{A}(-1, 0),\RRR\tau_{*}\circ\LLL\tau^{\prime*}\derived^\bounded(X').
  \end{equation}
  However, the first~$\ell$ terms are zero by \eqref{eqn:Blup1}, and the next~$k-\ell$ terms are in fact~$\Phi_{-k+\ell}(\derived^\bounded(F)), \dots, \Phi_{-1}(\derived^\bounded(F))$ by \cref{lemma:PhiA}. We get the desired generation.
\end{proof}

\begin{proof}[of \cref{prop:Generation}]
  Let~$\mathcal{C}\subset \derived^\bounded(\tilde X)$ be the triangulated subcategory generated by~\eqref{eqn:span}.

  We proceed by descending induction on~$m$. The induction hypothesis is that
  \begin{equation}
    \label{equation:induction-hypothesis}
    \mathcal{A}(m+1, 0), \dots, \mathcal{A}(-1, 0)\subset \mathcal{C}.
  \end{equation}
  The base case, which is when~$m=-k+\ell-1$, holds by assumption.
  For all~$a\in [m+1, -1]$, we have that
  \begin{equation}
    \mathcal{A}(a, -\ell), \dots,  \mathcal{A}(a, -1)\subset \mathcal{C}.
  \end{equation}
  Combined with the induction hypothesis \eqref{equation:induction-hypothesis} that~$\mathcal{A}(a, 0)\subset \mathcal{C}$, and with~\eqref{eqn:Astar1}, we have
  \begin{equation}
    \label{eqn:AstarInC}
    \mathcal{A}(a, \star)\subset \mathcal{C} \text{ for all } a\in [m+1, -1].
  \end{equation}
  Recall that similarly to \cref{remark:mutation-sod}, but now for the classical case of Orlov's blowup formula (after suitable mutations), we have also the semiorthogonal decomposition
  \begin{equation}
    \derived^\bounded(\tilde X)=\langle \mathcal{A}(m+1, \star), \dots, \mathcal{A}(-1, \star), \LLL\tau'^{*}\derived^\bounded(X'), \mathcal{A}(0, \star), \dots, \mathcal{A}(m+k, \star)\rangle.
  \end{equation}
  Thus for any object~$\mathcal{E}\in \mathcal{A}(m, 0)\subset\derived^\bounded(\tilde X)$, there is a distinguished triangle
  \begin{equation}
    \label{eqn:triangle-e}
    \mathcal{E}'\to \mathcal{E}\to \mathcal{E}''\xrightarrow{+1},
  \end{equation}
  with
  \begin{align}
    \label{eqn:E'}
    \mathcal{E}'&\in \langle \mathcal{A}(0, \star), \dots, \mathcal{A}(m+k, \star) \rangle \\
    \mathcal{E}''&\in \langle\mathcal{A}(m+1, \star), \dots, \mathcal{A}(-1, \star), \LLL\tau'^{*}\derived^\bounded(X') \rangle.
  \end{align}
  By \eqref{eqn:AstarInC}, $\mathcal{E}''\in \mathcal{C}$. Hence to show that~$\mathcal{E}\in \mathcal{C}$, it suffices to show that~$\mathcal{E}'\in \mathcal{C}$.

  We expand~\eqref{eqn:E'} using the projective bundle formula \eqref{eqn:Astar1} as
  \begin{equation}
    \label{equation:E'-sod}
    \begin{gathered}
      \mathcal{E}'\in\langle \mathcal{A}(0, -\ell-1), \dots, \mathcal{A}(0, -1);\mathcal{A}(1, -\ell-1), \dots, \mathcal{A}(1, -1); \\
      \dots;\mathcal{A}(m+k, -\ell-1), \dots, \mathcal{A}(m+k, -1)\rangle.
    \end{gathered}
  \end{equation}
  Let us denote by~$\mathcal{E}'(a, b)$ the projection of~$\mathcal{E}'$ to~$\mathcal{A}(a, b)$.
  For any~$(a, b)$ outside the rectangle described by~$a\in [0, m+k]$ and~$b\in [-\ell-1, -1]$,
  we have~$\mathcal{E}'(a,b)=0$ by \eqref{equation:E'-sod}.

  \paragraph{Claim}. {\ }
  We have that
  \begin{equation}
    \label{eqn:Eab}
    \mathcal{E}'(a, b)=0 \text{ for any } a\in [0, m+k]  \text{ and } b\in [-\ell-1, -\ell-1+a].
  \end{equation}
  \begin{proof}[of the claim]
    The proof is an induction on both~$a$ and~$b$,
    with descending order on~$a$ and with, for fixed~$a$, ascending order on~$b$.
    Recall that~$\mathcal{E}'(a,b)=0$ whenever~$(a,b)$ is outside the range used in \eqref{equation:E'-sod}.
    This will serve as the base case for our induction.

    Assuming that~$\mathcal{E}'(a', b')=0$ for all~$a'>a$, or~$a'=a$ and~$b'<b$, let us show $\mathcal{E}'(a, b)=0$. We have that
    \begin{equation}
      \RHom_{\tilde X}(\mathcal{E}'', \mathcal{A}(a-k, b+1))=0.
    \end{equation}
    Indeed, on one hand,~$\mathcal{E}''\in \langle\mathcal{A}(m+1, \star), \dots, \mathcal{A}(-1, \star), \LLL\tau'^{*}\derived^\bounded(X') \rangle$; on the other hand, $\mathcal{A}(a-k, b+1)\in \langle \mathcal{A}(-k, \star), \dots, \mathcal{A}(m, \star) \rangle$, since $a-k\in [-k, m]$. The vanishing follows from the semiorthogonality in \eqref{eqn:Blup2}.

    The second required vanishing result is that
    \begin{equation}
      \RHom_{\tilde X}(\mathcal{E}, \mathcal{A}(a-k, b+1))=0.
    \end{equation}
    Indeed, $\mathcal{E}\in \mathcal{A}(m, 0)$ and the vanishing follows from \cref{vanishing}. To see this, observe that we always have~$a-k\leq m$. If $a-k<m$, it is fine by \cref{vanishing}. If $a-k=m$, then $b+1\in [-\ell, -\ell+m+k]\subset[-\ell, -1]$, which is again covered by \cref{vanishing}.

    Consequently, $\RHom_{\tilde X}(\mathcal{E}', \mathcal{A}(a-k, b+1))=0$. However, by \cref{vanishing}, a non-zero morphism from~$\mathcal{E}'(a', b') $ to~$\mathcal{A}(a-k, b+1)$ only exists when~$a'>a$, or~$a'=a$ and~$b'\leq b$. But by the induction hypothesis, $\mathcal{E}'(a', b')=0$ for all~$a'>a$, or~$a'=a$ and~$b'<b$. Therefore, we deduce that
    \begin{equation}
      \label{eqn:Test}
      \RHom_X(\mathcal{E}'(a, b), \mathcal{A}(a-k, b+1))=0.
    \end{equation}
    Now write~$\mathcal{E}'(a,b)=j_{*}(p^{*}\circ\pi^{*}\mathcal{F}\otimes \mathcal{O}(a,b))$ for some~$\mathcal{F}\in \derived^\bounded(F)$. We define
    \begin{equation}
      \mathcal{G}\coloneqq j_{*}(p^{*}\circ\pi^{*}\mathcal{F}\otimes \mathcal{O}(a-k,b+1))[k+1],
    \end{equation}
    which is in~$\mathcal{A}(a-k, b+1)$. Making use of the triangle \eqref{eq:jtriangle} and
    \begin{equation}
      \RHom_E(p^{*}\circ\pi^{*}\mathcal{F}\otimes \mathcal{O}(a,b), p^{*}\circ\pi^{*}\mathcal{F}\otimes \mathcal{O}(a-k,b+1)[k+1])=0,
    \end{equation}
    we find
    \begin{equation}
      \begin{aligned}
        &\RHom_{\tilde X}(\mathcal{E}'(a, b), \mathcal{G}) \\
        &\qquad\cong\RHom_E(\LLL j^{*}\circ j_{*}(p^{*}\circ\pi^{*}\mathcal{F}\otimes \mathcal{O}(a,b)), p^{*}\circ\pi^{*}\mathcal{F}\otimes \mathcal{O}(a-k,b+1)[k+1])\\
        &\qquad\cong\RHom_E(p^{*}\circ\pi^{*}\mathcal{F}\otimes \mathcal{O}(a+1,b+1))[1], p^{*}\circ\pi^{*}\mathcal{F}\otimes \mathcal{O}(a-k,b+1)[k+1])
      \end{aligned}
    \end{equation}
    which contains the element~$\psi\coloneqq \operatorname{id}_{p^{*}\circ\pi^{*}\mathcal{F}}\otimes p^{*}(\alpha)\otimes p'^{*}(\operatorname{id}_{\mathcal{O}_{\pi'}(b+1)})$, where~$\alpha$ is a non-zero element in
    \begin{equation}
      \begin{aligned}
        \RHom_Z(\mathcal{O}_{\pi}(a+1), \mathcal{O}_{\pi}(a-k)[k])
        &\cong\HH^{k}(Z, \mathcal{O}_{\pi}(-k-1)) \\
        &\cong\HH^{0}(F, \RR^{k}\pi_{*}\mathcal{O}_{\pi}(-k-1)) \\
        &\cong \groundfield.
      \end{aligned}
    \end{equation}
    Then \eqref{eqn:Test} implies that~$\psi$ must be zero. This is only possible if~$p^{*}\circ\pi^{*}\mathcal{F}$ is itself zero.\checkthis{Sanity check: what exactly is used to deduce that $id_{p^{*}\circ\pi^{*}\mathcal{F}}=0$? A property of tensoring over $\mathcal{O}_E$ or a property of tensoring with line bundles?} Hence~$\mathcal{E}'(a, b)=0$. The induction for \eqref{eqn:Eab} is complete, which finishes the proof of the claim.
  \end{proof}
  In particular, the components~$\mathcal{E}'(a, -\ell-1)$ of~$\mathcal{E}'$ vanish for~$a\in [0, m+k]$. Therefore~$\mathcal{E}'\in \langle\mathcal{A}(\star, -\ell), \dots, \mathcal{A}(\star, -1)\rangle$, hence is contained in~$\mathcal{C}$. By \eqref{eqn:triangle-e}, this implies~$\mathcal{E}\in \mathcal{C}$. In other words,~$\mathcal{A}(m, 0)\subset \mathcal{C}$. The induction on~$m$ is complete and the proposition is proved.
\end{proof}

\section{Application: the Galkin--Shinder--Voisin diagram for cubic hypersurfaces}
\label{section:quadratic-fano-correspondence}
In this section we show there is a standard flip diagram associated to a smooth cubic hypersurface, so that \cref{thm:cubic} follows from \cref{theorem:flip}. In \cref{corollary:chow}, we use this result and \cite[theorem 3.4]{1910.06730v1} to deduce a corresponding isomorphism of Chow motives, generalizing \cite[theorem 5]{MR3639654}.

We now recall the setup of the Galkin--Shinder--Voisin diagram for cubic hypersurfaces. Let~$n\geq 2$, and let~$Y\subset\PP^{n+1}$ be a smooth cubic hypersurface. Denote~$F\coloneqq\fano(Y)$ the Fano variety of lines on~$Y$ and~$\hilbtwo[Y]$ the Hilbert scheme of two points on~$Y$, which will be discussed more explicitly in \cref{section:hilbert-squares}. Galkin--Shinder \cite[\S5]{1405.5154v2} and Voisin \cite{MR3646872} established the existence of the following diagram:

\begin{equation}
  \label{diag:cubic}
  \begin{tikzcd}
    & & E \arrow[d, hook, "j"] \arrow[ddll, swap, "p"] \arrow[ddrr, "p'"] \\
    & & D \arrow[dl, swap, "\tau"] \arrow[dr, "\tau'"] \\
    P_2 \arrow[r, hook, "i"] \arrow[rrd, swap, "\pi"] & \hilbtwo[Y] \arrow[dashed, rr, "\phi"] & & P_Y & P \arrow[l, hook', swap, "i'"] \arrow[dll, "\pi'"] \\
    & & F
  \end{tikzcd}
\end{equation}

In this diagram, we use the following notation:
\begin{itemize}
  \item $\pi'\colon P\to F$ is the universal~$\PP^{1}$-bundle;
  \item $\pi\colon P_{2}\coloneqq P^{[2]/F}\to F$ is the relative Hilbert square of~$\pi'$, which is a~$\mathbb{P}^2$\dash bundle;
  \item $P_{Y}\coloneqq \PP(\tangent_{\PP^{n+1}}|_{Y})\to Y$ is the projectivization of the restriction to~$Y$ of the tangent bundle of the projective space~$\PP^{n+1}$, which is a~$\PP^{n}$\dash bundle over~$Y$, and~$P_{Y}$ parametrises a point on~$Y$ together with a line in~$\PP^{n+1}$ passing through it;
  \item $D\coloneqq \{(L, z, x )\in \Gr(2,n+2)\times \hilbtwo[Y]\times Y \mid z\subset L, x\in L, \text{ and if } L\not\subset Y, z+x=Y\cap L\}$, \newline where~$\Gr(2,n+2)=\Gr(\mathbb{P}^1,\mathbb{P}^{n+1})$;
  \item given a point~$(L, z, x)$ of~$D$, the morphism~$\tau$ maps it to~$z\in \hilbtwo[Y]$ while the morphism~$\tau'$ maps it to~$(L, x)\in P_{Y}$;
  \item $E\coloneqq \{(L, z, x )\in D \mid L\subset Y\}$ is a divisor in $D$;
  \item $\phi$ is the Galkin--Shinder map \cite[(5.3)]{1405.5154v2} which sends~$z\in \hilbtwo[Y]$ to~$(L, x)\in P_{Y}$, where $L$ is the line determined (i.e.~spanned) by~$z$, and~$x$ is the residual intersection point of~$L$ with~$Y$, i.e.~$L\cap Y=z+x$, when~$L$ is not completely contained in~$Y$;
  \item by Voisin \cite[proposition~2.9]{MR3646872} the upper two trapezoids are blowup diagrams, and in particular all varieties appearing above are smooth and projective;
  \item the outer square is cartesian, and observe that~$\pi$ and~$p'$ are~$\mathbb{P}^{2}$\dash bundles while~$\pi'$ and~$p$ are~$\mathbb{P}^{1}$-bundles.
\end{itemize}
Denote by~$\mathcal{S}$ the restriction of the tautological rank-2 bundle over the Grassmannian~$\Gr(2,n+2)$ to~$F$.

In \cref{corollary:quadratic-fano-standard-flip} we will show that this diagram is in fact a standard flip diagram, as in~\eqref{equation:flip}.

\begin{lemma}
  \label{lem:disjoint}
  The natural morphism~$i\colon P_{2}\to Y^{[2]}$ is a closed immersion.
\end{lemma}

\begin{proof}
  Let us describe the morphism~$i$ on closed points. Let~$[L]\in F$ be a point representing (by abuse of notation) a line~$L\subset Y$. Then for two distinct points~$x, y\in L$, the morphism~$i$ sends~$x+y$ to the length-2 subscheme~$\{x, y\}\in Y^{[2]}$; while for the double point~$2x\in L^{[2]}$, it sends it to the length-2 subscheme of $Y$ supported on $x$ with the tangent direction given by the 1-dimensional subspace~$\tangent_{x}L\subset\tangent_{x}Y$.

  Now observe that if~$L_1\neq L_2$ are two distinct points in~$F$, then~$\hilbtwoLone\cap\hilbtwoLtwo=\emptyset$. To see this, we discuss the two possibilities:
  \begin{itemize}
    \item if~$L_1\cap L_2=\emptyset$, then this is immediate;
    \item if~$L_1\cap L_2=\{z\}$, then points on~$\hilbtwoLone$ and~$\hilbtwoLtwo$ whose support contains~$z$ are distinguished by the support of the second point, or the tangent direction.
  \end{itemize}
  We leave the reader to check the injectivity of the map between the tangent spaces.
\end{proof}


\begin{lemma}
\label{lem:NormalBundle}
The normal bundle of $i$ is of the form
\begin{equation}
	 \normal_{P_{2}/\hilbtwo[Y]}\cong \mathcal{O}_\pi(-1)\otimes\pi^*(\mathcal{V}'),
\end{equation}
for some rank-2 vector bundle $\mathcal{V}'$ on $F$.
\end{lemma}

\begin{proof}
	It is enough to show that when restricted to each fiber of $\pi$, the vector bundle bundle $\normal_{P_{2}/\hilbtwo[Y]}$ is a direct sum of two copies of the restriction of $\mathcal{O}_{\pi}(-1)$. We will use the construction of Voisin in the proof of \cite[proposition~2.9]{MR3646872}, which we recall now. Denote~$\Gr\coloneqq\Gr(2, n+2)$ the Grassmannian of lines in~$\mathbb{P}^{n+1}$. Let~$\sigma\colon Q\to \Gr$ be the universal~$\PP^1$\dash bundle associated to the tautological rank-2 vector bundle.

Pulling back the~$\PP^1$\dash bundle using the natural morphism~$\hilbtwo[Y]\to \Gr$, we obtain the morphism~$\sigma_1\colon Q_1\to \hilbtwo[Y]$. Denote the evaluation morphism~$e_{1}\colon Q_{1}\to \PP^{n+1}$. Then the preimage~$e_1^{-1}(Y)$ is a divisor in~$Q_{1}$ of degree 3 over~$\hilbtwo[Y]$, with one component~$D_1$ being the universal subscheme and the other component exactly~$D$. We can summarise the notation as
  \begin{equation}
    \begin{tikzcd}
      & Y \arrow[r, hook] & \PP^{n+1} \\
      D_1\cup D \arrow[r, equals] & e_1^{-1}(Y) \arrow[u] \arrow[r, hook] & Q_1 \arrow[u, swap, "e_1"] \arrow[r] \arrow[d, swap, "\sigma_1"] & Q \arrow[d, "\sigma"] \\
      & & \hilbtwo[Y] \arrow[r] & \Gr
    \end{tikzcd}.
  \end{equation}
  Then~$P_2$ is the zero set of a regular section of the rank-2 vector bundle~$\sigma_{1, *}(\mathcal{O}_{\sigma_{1}}(3)\otimes \mathcal{O}_{Q_{1}}(-D_{1}))$. Therefore the normal bundle of~$P_{2}$ inside~$\hilbtwo[Y]$ is its restriction to~$P_{2}$, which is
 \begin{equation}
   \label{eqn:NormalBundle1}
    \pr_{1,*}\left( \pr_2^*\mathcal{O}_\varpi(3)\otimes\mathcal{O}(-\mathcal{U}) \right)
  \end{equation}
  where~$\mathcal{U}$ is the universal subscheme (i.e.~the restriction of~$D_1$) in the diagram
\begin{equation}
    \begin{tikzcd}
      \mathcal{U} \arrow[r, hook] \arrow[rd] & P_{2}\times_{F} \PP(\mathcal{S}) \arrow[r, "\pr_2"] \arrow[d, "\pr_1"] & \PP(\mathcal{S}) \arrow[d, "\varpi"]\\
      & P_{2} \arrow[r, "\pi"] & F
    \end{tikzcd}
  \end{equation}
  and where $\mathcal{S}$ denotes the tautological rank two vector bundle on $\Gr$. Let us compute the restriction of the vector bundle \eqref{eqn:NormalBundle1} to each fiber of $\pi$.

  For any~$[L]\in F$, the bundle \eqref{eqn:NormalBundle1} restricted to~$\pi^{-1}(L)\cong\hilbtwo[L]$ is given by
  \begin{equation}
    \label{equation:Nprime-identification}
    \pr_{1,*}\left( \pr_2^*\mathcal{O}_L(3)\otimes\mathcal{O}_{\hilbtwo[L]\times L}(-U) \right)
  \end{equation}
  where~$U$ is the universal subscheme (the fiber of $D_1$ over $L$) in the diagram
  \begin{equation}
    \begin{tikzcd}
      U \arrow[r, hook] \arrow[rd] & \hilbtwo[L]\times L \arrow[r, "\pr_2"] \arrow[d, "\pr_1"] & L \\
      & \hilbtwo[L]
    \end{tikzcd}
  \end{equation}
  and~$U\cong L\times L$, so that~$U\to\hilbtwo[L]$ is the quotient by the involution and~$U\to L$ is the projection on the first factor. We obtain that
  \begin{equation}
    \mathcal{O}_{\hilbtwo[L]\times L}(U)\cong\mathcal{O}_{\mathbb{P}^2\times\mathbb{P}^1}(1,2).
  \end{equation}
  Indeed, in the Chow ring of~$\mathbb{P}^2\times\mathbb{P}^1$ we compute that~$U\cdot(\{z\}\times L)=2$ for~$z\in\hilbtwo[L]$, and we have that~$U|_{\hilbtwo[L]\times\{x\}}\cong\mathcal{O}_{\hilbtwo[L]}(1)$ for~$x\in L$.

  As a result, we can compute the vector bundle  \eqref{equation:Nprime-identification} as
    \begin{equation}
  \begin{aligned}
  \pr_{1,*}\left( \pr_2^*\mathcal{O}_L(3)\otimes\mathcal{O}_{\hilbtwo[L]\times L}(-U) \right)
  &\cong\pr_{1,*}\left( \pr_{1}^{*}\mathcal{O}_{\hilbtwo[L]}(-1)\otimes\pr_2^*\mathcal{O}_L(1) \right) \\
  &\cong\mathcal{O}_{\hilbtwo[L]}(-1)\otimes\pr_{1,*}\left(\pr_2^*\mathcal{O}_L(1) \right) \\
  &\cong\mathcal{O}_{\hilbtwo[L]}(-1)^{\oplus 2}.
  \end{aligned}
  \end{equation}
This completes the proof.
\end{proof}

\begin{remark}
	A precise formula for the normal bundle $\normal_{P_{2}/\hilbtwo[Y]}$ is recently obtained in Renjie Lyu's thesis \cite[\S 4.2]{LyuThesis}. In our notation, if we identify $P_2$ with $\PP(\Sym^2\mathcal{S})$, then $\mathcal{V}'\cong \mathcal{S}\otimes \mathcal{O}_F(3)$, where $\mathcal{O}_F(1)$ is the restriction of the Pl\"ucker polarisation on $\Gr(2, n+2)$.
\end{remark}
\begin{corollary}
  \label{corollary:quadratic-fano-standard-flip}
  The diagram \eqref{diag:cubic} is a standard flip diagram.
\end{corollary}

\begin{proof}
  Indeed, thanks to \cref{lem:NormalBundle}, we take in the notation of \eqref{equation:flip}, $X\coloneqq \hilbtwo[Y]$, $X'\coloneqq P_{Y}$, $Z\coloneqq P_{2}$, $Z'\coloneqq P$, $D\coloneqq \tilde X$, $k=2$, $\ell=1$, and $\mathcal{V}=\Sym^{2}\mathcal{S}$.
\end{proof}

Hence we obtain the following.
\begin{proof}[of \cref{thm:cubic}]
  By \cref{corollary:quadratic-fano-standard-flip}, we can apply \cref{theorem:flip} and Remark \ref{remark:mutation-sod} to the diagram~\eqref{diag:cubic} to get semiorthogonal decompositions
  \begin{equation}
      \derived^\bounded(\hilbtwo[Y])
      =\langle\Phi_{-1}(\derived^\bounded(F)),\RRR\tau_{*}\circ\LLL\tau^{\prime*}\derived^\bounded(P_{Y})\rangle
  \end{equation}
  where~$\Phi_{m}(-)=i_{*}(\pi^{*}(-)\otimes \mathcal{O}_{\pi}(m))\colon\derived^\bounded(F)\to \derived^\bounded(\hilbtwo[Y])$ is the fully faithful functor from \eqref{equation:Phi-definition}.

  We can now further decompose the~$\PP^n$\dash bundle~$P_Y\to Y$ using the projective bundle formula, and conclude.
\end{proof}


\begin{remark}[(Grothendieck ring of categories)]
  \label{remark:K0-cats}
  The~$Y$-$\fano(Y)$\dash relation from \cite[theorem~5.1]{1405.5154v2} is the equality
  \begin{equation}
    [\hilbtwo[Y]]=[\mathbb{P}^n][Y]+\mathbb{L}^2[\fano(Y)]
  \end{equation}
  in the Grothendieck ring of varieties~$\Kzero(\Var/\groundfield)$, for a cubic hypersurface~$Y\hookrightarrow\mathbb{P}^{n+1}$.

  Taking the motivic measure in the Grothendieck ring of dg~categories~$\Kzero(\dgCat/\groundfield)$ from \cite{MR2051435} gives the relation
  \begin{equation}
    [\derived^\bounded(\hilbtwo[Y])]=(n+1)[\derived^\bounded(Y)]+[\derived^\bounded(\fano(Y))],
  \end{equation}
  which is suggestive of a semiorthogonal decomposition of the form \eqref{equation:sod-quadratic-fano}, and \cref{thm:cubic} indeed provides this.
\end{remark}

Using Jiang's analysis of the Chow theory of a standard flip diagram in \cite{1910.06730v1}, we can obtain the following improvement of Laterveer's result in \cite[theorem~5]{MR3639654} to integral coefficients. The same result is independently obtained in Renjie Lyu's upcoming thesis \cite[\S 4]{LyuThesis} without resorting to \cite{1910.06730v1}. For a smooth scheme~$X$ over some ground field~$\groundfield$, we write~$\mathfrak{h}(X)=(X, \Delta_{X})$ for the motive of $X$ in the rigid tensor category of Chow motives over~$\groundfield$, see \cite{MR3052734}. By~$\mathfrak{h}(X)(-i)$ we denote the motive~$\mathfrak{h}(X) \otimes \mathbb{L}^i$, where~$\mathbb{L} = (\PP^1,[\PP^1 \times \{0\}])$ is the Lefschetz motive.

\begin{corollary}[(An isomorphism of Chow motives)]
  \label{corollary:chow}
  For an arbitrary base field~$\groundfield$, there is an isomorphism of Chow motives (with integral coefficients) over~$\groundfield$:
  \begin{equation}
    \mathfrak{h}(F)(-2) \oplus \bigoplus_{i=0}^{n} \mathfrak{h}(Y)(-i) \xrightarrow{\cong} \mathfrak{h}(Y^{[2]}),
  \end{equation}
  where the isomorphism is given by
  \begin{equation}
    \left(i_{*}\circ \pi^{*}, \bigoplus_{i=0}^n\tau_{*}\circ \tau'^{*}\circ \cdot \mathrm{c}_{1}(\mathcal{O}_{u}(1))^{i}\circ u^{*}\right)
  \end{equation}
  where~$u\colon P_{Y}\to Y$ the~$\PP^{n}$-bundle and the other notation is as in \eqref{diag:cubic}.
\end{corollary}

\begin{proof}
  It suffices to combine \cref{corollary:quadratic-fano-standard-flip} with \cite[corollary 3.8]{1910.06730v1}.
\end{proof}

\section{Hilbert squares of low-degree hypersurfaces}
\label{section:hilbert-squares}
Let~$X$ be a smooth projective variety of dimension~$n=\dim X$. Then its Hilbert square (which is short-hand for the Hilbert scheme of length-2 subschemes) $\hilbtwo$ is a smooth projective variety of dimension~$2n$ \cite[theorem~3.0.1]{MR1616606}. Throughout we will assume that~$n\geq 2$, with the main focus being on~$n\geq 3$. For the benefit of the reader we will recall its relevant properties.

It has the following explicit description:
\begin{equation}
  \begin{tikzcd}
    \hilbtwo \\
    \Bl_\Delta(X\times X) \arrow[u, "q"] \arrow[d, swap, "\tau"] & E \arrow[d] \arrow[l, hook']  \arrow[dl, phantom, "\square"] \\
    X\times X & X \arrow[l, hook', "\Delta"]
  \end{tikzcd}
  \label{equation:double-cover-notation}
\end{equation}
where the square is the cartesian blowup square, and~$q$ is the quotient by the~$\mathbb{Z}/2\mathbb{Z}$\dash action on~$\Bl_\Delta(X\times X)$.

We can consider~$q$ as a double cover, and we need to describe it more explicitly. We have that
\begin{equation}
  \Pic(\hilbtwo)\cong\Pic(X)\oplus\mathbb{Z}[\delta]
\end{equation}
for the divisor~$\delta$ which is half of the divisor~$D$ on~$\hilbtwo$ parametrising non-reduced length-2 subschemes. The double cover is branched along~$D$, and we have the following picture
\begin{equation}
  \begin{tikzcd}
    E \arrow[r, hook] \arrow[rd, swap, "\cong"] & 2E \arrow[d] \arrow[r, hook] \arrow[dr, phantom, "\square"] & \Bl_\Delta(X\times X) \arrow[d, "q"] \\
    & D \arrow[r, hook] & \hilbtwo
  \end{tikzcd}
\end{equation}
where~$2E$ is the second infinitesimal neighbourhood of the exceptional divisor~$E$.

By \cite[lemma~I.17.1]{MR2030225} we have that the description as a double cover gives the identifications
\begin{align}
  \omega_{\Bl_\Delta(X\times X)}&\cong q^*\left( \omega_{\hilbtwo}\otimes\mathcal{O}_{\hilbtwo}(\delta) \right) \label{equation:double-cover-canonical} \\
  q^*(\mathcal{O}_{\hilbtwo}(\delta))&\cong\mathcal{O}_{\Bl_\Delta(X\times X)}(E) \label{equation:double-cover-divisor}
\end{align}
whilst the description as a blowup gives the identification
\begin{equation}
  \label{equation:blowup-canonical}
  \omega_{\Bl_\Delta(X\times X)}\cong\tau^*(\omega_{X\times X})\otimes\mathcal{O}_{\Bl_\Delta(X\times X)}((n-1)E),
\end{equation}
so that
\begin{equation}
  q^*(\omega_{\hilbtwo})\cong\tau^*(\omega_{X\times X})\otimes\mathcal{O}_{\Bl_\Delta(X\times X)}((n-2)E).
\end{equation}

The Hilbert square is a resolution of singularities, via the Hilbert--Chow morphism $\hilbtwo\to\Sym^2X$, sending a length-2 subscheme to the cycle it is supported on. The resolution is crepant if~$n=2$, but fails to be crepant for~$n\geq 3$. To see this, observe that the Hilbert--Chow morphism fits in the commutative diagram
\begin{equation}
  \begin{tikzcd}
    \Bl_\Delta(X\times X) \arrow[r, "\tau"] \arrow[d, swap,"q"] & X\times X \arrow[d] \\
    \hilbtwo \arrow[r] & \Sym^2X
  \end{tikzcd}
\end{equation}
which allows us to compare the canonical bundles. We conclude that it is crepant if and only if the multiplicity of the contribution of~$\delta$ in the double cover agrees with that of the exceptional divisor in the blowup, which happens if and only if~$n=2$.

We will now use this description of the Hilbert square to prove that in some cases the Hilbert square of a Fano variety is again Fano. This possibility was suggested in \cite{315322}, and later used in \cite[lemma~2.9]{2001.06078v1}, for~$X=\mathbb{P}^n$ where~$n\geq 3$. We will amplify this result to also cover low-degree hypersurfaces.

\subsection{Projective space}
We will recall the setup and proof of \cite[lemma~2.9]{2001.06078v1}. We will use the notation of \eqref{equation:double-cover-notation} for~$X=\mathbb{P}^n$. Consider the morphism
\begin{equation}
  \varphi\colon\Bl_\Delta(\mathbb{P}^n\times\mathbb{P}^n)\to\mathbb{P}^n\times\mathbb{P}^n\times\Gr(2,n+1)
\end{equation}
which is the product of the blowup~$\tau$ and the morphism~$\Bl_\Delta(\mathbb{P}^n\times\mathbb{P}^n)\to\Gr(2,n+1)$ sending a point on the blowup to the line in~$\mathbb{P}^n$ spanned by two distinct points or a point together with a tangent direction, which gives a point on the Grassmannian~$\Gr(2,n+1)$. It is quasi-finite by construction (its fibres are either empty or singletons), hence by \cite[\href{https://stacks.math.columbia.edu/tag/02LS}{tag 02LS}]{stacks} it is actually finite (and a computation on the tangent spaces shows it is in fact a closed immersion, but we will not need this).

We have that~$\Pic(\mathbb{P}^n\times\mathbb{P}^n\times\Gr(2,n+1))\cong\mathbb{Z}^{\oplus3}$ and we will denote
\begin{equation}
  \mathcal{O}_{\mathbb{P}^n\times\mathbb{P}^n\times\Gr(2,n+1)}(i,j,k)
  \coloneqq
  \mathcal{O}_{\mathbb{P}^n}(i)\boxtimes\mathcal{O}_{\mathbb{P}^n}(j)\boxtimes\mathcal{O}_{\Gr(2,n+1)}(k),
\end{equation}
here~$\mathcal{O}_{\Gr(2,n+1)}(1)$ is the Pl\"ucker line bundle. The line bundles~$\mathcal{O}_{\mathbb{P}^n\times\mathbb{P}^n\times\Gr(2,n+1)}(i,j,k)$ are ample if and only if~$i,j,k\geq 1$, and because~$\varphi$ is finite (so in particular affine) ampleness is preserved under pullback.

We have the identifications
\begin{equation}
  \label{equation:identifications-varphi}
  \begin{aligned}
    \varphi^*(\mathcal{O}_{\mathbb{P}^n\times\mathbb{P}^n\times\Gr(2,n+1)}(1,1,0))
    &\cong\tau^*\left( \mathcal{O}_{\mathbb{P}^n\times\mathbb{P}^n}(1,1) \right) \\
    \varphi^*(\mathcal{O}_{\mathbb{P}^n\times\mathbb{P}^n\times\Gr(2,n+1)}(0,0,1))
    &\cong\tau^*\left( \mathcal{O}_{\mathbb{P}^n\times\mathbb{P}^n}(1,1) \right)\otimes\mathcal{O}_{\Bl_\Delta(\mathbb{P}^n\times\mathbb{P}^n)}(-E),
  \end{aligned}
\end{equation}
where the second isomorphism is obtained by considering Pl\"ucker coordinates in terms of bilinear forms (so as sections of~$\mathcal{O}_{\mathbb{P}^n\times\mathbb{P}^n}(1,1)$) which are non-vanishing on the diagonal. We then obtain the following proposition, which is also \cite[lemma~2.9]{2001.06078v1}.

\begin{proposition}[(Sawin)]
  \label{proposition:sawin}
  Let~$n\geq 3$. Then~$\hilbtwoP$ is Fano.
\end{proposition}

\begin{proof}
  By \eqref{equation:double-cover-canonical} and \eqref{equation:double-cover-divisor} we have the identification
  \begin{equation}
    \begin{aligned}
      &q^*(\omega_{\hilbtwoP}^\vee) \\
      &\quad\cong\omega_{\Bl_\Delta(\mathbb{P}^n\times\mathbb{P}^n)}^\vee\otimes q^*(\mathcal{O}_{\hilbtwoP}(\delta)) \\
      &\quad\cong\omega_{\Bl_\Delta(\mathbb{P}^n\times\mathbb{P}^n)}^\vee\otimes\mathcal{O}_{\Bl_\Delta(\mathbb{P}^n\times\mathbb{P}^n)}(E).
    \end{aligned}
  \end{equation}
  By \eqref{equation:blowup-canonical} we can further rewrite this as
  \begin{equation}
    \quad\cong\tau^*\left( \mathcal{O}_{\mathbb{P}^n\times\mathbb{P}^n}(n+1,n+1) \right)\otimes\mathcal{O}_{\Bl_\Delta(\mathbb{P}^n\times\mathbb{P}^n)}((2-n)E).
  \end{equation}
  But using \eqref{equation:identifications-varphi} we can realise this line bundle as
  \begin{equation}
    \quad\cong\varphi^*\left( \mathcal{O}_{\mathbb{P}^n\times\mathbb{P}^n\times\Gr(2,n+1)}(3,3,n-2) \right)
  \end{equation}
  and hence it is ample when~$n\geq 3$.

  To conclude, by \cite[\href{https://stacks.math.columbia.edu/tag/0B5V}{tag 0B5V}]{stacks} we know that~$\omega_{\hilbtwoP}^\vee$ is ample if and only if~$q^*(\omega_{\hilbtwoP}^\vee)$ is ample, and we are done.
\end{proof}

\subsection{Quadrics and cubic hypersurfaces}
We will now bootstrap from the result in \cref{proposition:sawin} to show that low-degree hypersurfaces again give rise to Hilbert squares which are Fano.

We will extend the notation of \eqref{equation:double-cover-notation} for~$X\hookrightarrow\mathbb{P}^{n+1}$ a hypersurface of degree~$d=2,3$ by decorating the morphisms involving~$\mathbb{P}^{n+1}$. By functoriality of the Hibert scheme construction we obtain the following diagram.
\begin{equation}
  \begin{tikzcd}
    & \hilbtwo \arrow[r, hook] & \hilbtwoP[n+1] \\
    E \arrow[r, hook] \arrow[d] & \Bl_\Delta(X\times X) \arrow[r, hook, "\imath"] \arrow[u, "q"] \arrow[d, swap, "\tau"] & \Bl_\Delta(\mathbb{P}^{n+1}\times\mathbb{P}^{n+1}) \arrow[u, swap, "q'"] \arrow[d, "\tau'"] & E' \arrow[l, swap, hook'] \arrow[d] \\
    X\arrow[r, hook, "\Delta"] & X\times X \arrow[r, hook] & \mathbb{P}^{n+1}\times\mathbb{P}^{n+1} & \mathbb{P}^{n+1} \arrow[l, hook', swap, "\Delta'"]
  \end{tikzcd}.
\end{equation}

\begin{proposition}
  \label{proposition:fano-hilbert-square-hypersurface}
  Let~$n\geq 3$ and~$d=1,2,3$. Let~$X\subseteq\mathbb{P}^{n+1}$ be a hypersurface of degree~$d$. Then~$\hilbtwo$ is Fano.
\end{proposition}

\begin{proof}
  We will write~$\mathcal{O}_{X\times X}(i,j)\coloneqq\mathcal{O}_X(i)\boxtimes\mathcal{O}_X(j)$. By the adjunction formula we have~$\omega_X\cong\mathcal{O}_X(-n-2+d)$. As in the proof of \cref{proposition:sawin}, using \eqref{equation:double-cover-canonical}, \eqref{equation:double-cover-divisor} and \eqref{equation:blowup-canonical} we have the isomorphism
  \begin{equation}
    q^*\left( \omega_{\hilbtwo} \right)
    \cong
    \tau^*(\mathcal{O}_{X\times X}(-n-2+d,-n-2+d))\otimes\mathcal{O}_{\Bl_\Delta(X\times X)}((n-2)E).
  \end{equation}
  Now, by the blowup closure lemma from \cite[lemma~22.2.6]{vakil} we get that the exceptional divisor~$E$ of~$\Bl_\Delta(X\times X)$ is the restriction of the exceptional divisor~$E'$ on~$\Bl_{\Delta'}(\mathbb{P}^{n+1}\times\mathbb{P}^{n+1})$, so we have an isomorphism
  \begin{equation}
    \begin{aligned}
      &\tau^*(\mathcal{O}_{X\times X}(n-2,n-2))\otimes\mathcal{O}_{\Bl_\Delta(X\times X)}((2-n)E) \\
      &\quad\cong\imath^*\left( \tau'^*(\mathcal{O}_{\mathbb{P}^{n+1}\times\mathbb{P}^{n+1}}(n-2,n-2))\otimes\mathcal{O}_{\Bl_{\Delta'}(\mathbb{P}^{n+1}\times\mathbb{P}^{n+1})}((2-n)E') \right).
    \end{aligned}
  \end{equation}
  Next we consider the composition of the morphism~$\varphi$ (now for~$\mathbb{P}^{n+1}$) with the inclusion~$\imath$, which allows us to further rewrite this line bundle as
  \begin{equation}
    \quad\cong\imath^*\circ\varphi^*\left( \mathcal{O}_{\mathbb{P}^{n+1}\times\mathbb{P}^{n+1}\times\Gr(2,n+2)}(0,0,n-2) \right).
  \end{equation}
  Combining this with the isomorphism
  \begin{equation}
    \tau^*\left( \mathcal{O}_{X\times X}(4-d,4-d) \right)\cong\imath^*\circ\varphi^*\left( \mathcal{O}_{\mathbb{P}^{n+1}\times\mathbb{P}^{n+1}\times\Gr(2,n+2)}(4-d,4-d,0) \right)
  \end{equation}
  we obtain after dualising the isomorphism
  \begin{equation}
    q^*\left( \omega_{\hilbtwo}^\vee \right)
    \cong
    \imath^*\circ\varphi^*\left( \mathcal{O}_{\mathbb{P}^{n+1}\times\mathbb{P}^{n+1}\times\Gr(2,n+2)}(4-d,4-d,n-2) \right).
  \end{equation}
  Hence we see that~$q^*(\omega_{\hilbtwo}^\vee)$ is ample as soon as~$d\leq 3$ and~$n\geq 3$, because~$\varphi\circ\imath$ is a again finite morphism.

  We can now conclude as in the proof of \cref{proposition:sawin}, as by \cite[\href{https://stacks.math.columbia.edu/tag/0B5V}{tag 0B5V}]{stacks} we know that~$\omega_{\hilbtwo}^\vee$ is ample if and only if~$q^*(\omega_{\hilbtwo}^\vee)$ is ample.
\end{proof}

\begin{remark}
  The Hilbert square of a quartic hypersurface~$X$ in~$\mathbb{P}^{n+1}$ is \emph{not} Fano. By inspecting the proof, we see that the anticanonical bundle on~$\hilbtwo$ is the pullback of the~$(n-2)$th power of the Pl\"ucker polarisation for~$\Gr(2,n+2)$ under a morphism which is generically finite (of degree~6) but not finite itself. Hence~$\hilbtwo$ is only weak Fano in this case.
\end{remark}

\begin{remark}[(Geometric properties of the Hilbert squares)]
  \label{remark:geometric-properties}

  By \cite[theorem~C]{MR3950704} we can understand the deformation theory of these Fano Hilbert squares. We have that~$\mathbb{P}^n$ and~$Q^n$ are rigid, and therefore so are~$\hilbtwo[\mathbb{P}^n]$ and~$\hilbtwo[Q^n]$. Cubic hypersurfaces of dimension~$n$ on the other hand come in a family of dimension~$\binom{n+2}{3}$, and so do their Hilbert squares.

  Finally, by \cite[theorem~4]{aut-hilbnS} we have that the automorphism group of~$\hilbtwo[\mathbb{P}^n]$ is isomorphic to that~$\mathbb{P}^n$, i.e.~is given by~$\PGL_{n+1}$. It would be interesting to understand the automorphism groups of Hilbert squares (and especially those which are again Fano) more generally.
\end{remark}

\subsection{Application: the Fano variety of lines is a Fano visitor}
\label{subsection:fano-visitor}
As discussed in the introduction, the \emph{Fano visitor problem} for a smooth projective variety~$X$ asks for the construction of a fully faithful functor~$\derived^\bounded(X)\hookrightarrow\derived^\bounded(Y)$ into the derived category of a smooth projective Fano variety~$Y$. In this case we call~$X$ a \emph{Fano visitor} and~$Y$ the \emph{Fano host}. We will study this for the Fano variety of lines on a cubic hypersurface.

Now let~$Y\subset\mathbb{P}^{n+1}$ be a smooth cubic hypersurface. By \cref{thm:cubic} we have that~$\derived^\bounded(\fano(Y))$ is an admissible subcategory of~$\derived^\bounded(\hilbtwo[Y])$, whilst~$\hilbtwo[Y]$ is Fano if~$n\geq 3$ by \cref{theorem:fano}. This proves \cref{corollary:fano-visitor}. In fact, if~$n\geq 5$ then~$\fano(Y)$ is itself a Fano variety, which is why we are only interested in~$n=3,4$ for the Fano visitor problem \cite[proposition~1.8]{MR0569043}.

\paragraph{Cubic threefolds}\quad
If~$n=3$, then~$\fano(Y)$ is a smooth projective surface of general type \cite[proposition~1.21]{MR0569043}, whose Hodge diamond is
\begin{equation}
  \begin{array}{cccccc}
    & & 1 \\
    & 5 & & 5 \\
    10 & & 25 & & 10 \\
    & 5 & & 5 \\
    & & 1
  \end{array}.
\end{equation}
As far as we know this is the first construction of a Fano host for a surface of general type which is not a complete intersection or a product of curves.

Observe that~$\omega_{\fano(Y)}\cong\mathcal{O}_{\fano(Y)}(1)$ is a very ample line bundle (as it is the restriction of~$\mathcal{O}_{\mathbb{P}^{\binom{n+2}{2}}-1}(1)$ along the closed immersions~$\fano(Y)\hookrightarrow\Gr(2,n+2)\hookrightarrow\mathbb{P}^{\binom{n+2}{2}-1}$), hence by \cite[corollary~1.5]{1508.00682v2} we have that~$\derived^\bounded(\fano(Y))$ is indecomposable.

\paragraph{Cubic fourfolds}\quad
If~$n=4$, then the Fano variety of lines on a cubic fourfold is a~4\dash dimensional hyperk\"ahler variety, deformation equivalent to the Hilbert square of a K3 surface \cite{MR0818549}. Hence we have constructed Fano hosts for the complete~20\dash dimensional family of hyperk\"ahlers arising from Fano varieties for cubic fourfolds. As~$\derived^\bounded(\fano(Y))$ is Calabi--Yau, it is indecomposable.

It is still an interesting question to find a Fano host for other~20\dash dimensional families of hyperk\"ahler~4\dash folds deformation equivalent to the Hilbert square of a~K3~surface.

\paragraph{Higher dimensions}\quad
If~$n\geq 5$, then the Fano variety of lines is itself a Fano variety, and the interesting question is to find a natural semiorthogonal decomposition for it, rather than embed it in the derived categoryof a Fano variety. We do not address this here.

\paragraph{Cubic surfaces}\quad
If~$n=2$, then~$\hilbtwo[Y]$ is not a Fano variety, because the Hilbert--Chow morphism is a crepant resolution of singularities of~$\Sym^2Y$, but it is log Fano \cite[corollary~3]{MR3114922}. \Cref{thm:cubic} still applies though, and we obtain a semiorthogonal decomposition
\begin{equation}
  \derived^\bounded(\hilbtwo[Y])
  =\langle\derived^\bounded(\fano(Y)),\derived^\bounded(Y),\derived^\bounded(Y),\derived^\bounded(Y)\rangle,
\end{equation}
which can be refined into a full exceptional collection of length~54.

As~$\fano(Y)$ consists of~27~points, we obtain~27 completely orthogonal objects in~$\derived^\bounded(\hilbtwo[Y])$. These are the structure sheaves of the~$\mathbb{P}^2$'s embedded in~$\hilbtwo[Y]$ which parametrise~2~points on each of the~27~lines. Alternatively, consider the natural morphism~$\hilbtwo[Y]\to\Gr(2,4)$ sending~2~points to the line they span in~$\mathbb{P}^3$. This morphism is generically finite of degree~$\binom{3}{2}=3$ onto its image, and the~27~completely orthogonal objects correspond to the locus where the morphism is not finite.


\newpage

\appendix

\section{Short proof of \texorpdfstring{\cref{theorem:flip}\ref{enumerate:part-2}}{Theorem A(ii)} for \texorpdfstring{$\ell=1$}{l=1}}
\label{app:l1}
In this appendix we give a short proof of \cref{theorem:flip}\ref{enumerate:part-2} when~$\ell=1$. This proof has the interesting property that it compares the category~$\derived^\bounded(X)$ to the components of \eqref{eqn:goal} as subcategories in~$\derived^\bounded(\tilde X)$, and not only after applying~$\RRR\tau_*$.

Assume we are given a standard flip diagram \eqref{equation:flip} with $\ell=1$. We recall some of the notation introduced in section~\ref{section:standard-flip}. Let~$a,b\in\mathbb{Z}$, then we denote by~$\mathcal{O}(a,b)\coloneqq p^*\mathcal{O}_\pi(a)\otimes p^{\prime*}\mathcal{O}_{\pi'}(b)$, which is a line bundle on~$E$. We will consider the following triangulated subcategory (with reasoning as for \eqref{equation:subcategory}) of~$\derived^\bounded(\tilde X)$:
\begin{equation}
  \begin{aligned}
    \mathcal{A}(a,b)
    &\coloneqq j_*(p^*\circ\pi^*(\derived^\bounded(F)\otimes\mathcal{O}(a,b)) \\
    &=j_*(p^{\prime*}\circ\pi^{\prime*}(\derived^\bounded(F)\otimes\mathcal{O}(a,b)).
  \end{aligned}
\end{equation}

As~$\tilde X$ is the blowup of~$X$ along~$Z$ (which has codimension~2 by assumption), Orlov's blowup formula \cite[theorem~4.3]{MR1208153} gives a semiorthogonal decomposition
\begin{equation}
  \label{eqn:BlowupZX}
  \derived^\bounded(\tilde X)
  =\langle j_*(p^*(\derived^\bounded(Z))\otimes\mathcal{O}_E(E)),\LLL\tau^*\derived^\bounded(X)\rangle.
\end{equation}
Using the fact that~$Z$ is a~$\PP^k$-bundle over~$F$ by \eqref{equation:flip}, the first component in \eqref{eqn:BlowupZX} can be further decomposed using Orlov's projective bundle formula \cite[theorem~2.6]{MR1208153} using our notation as
\begin{equation}
  \label{eq:sod-pbundle}
  j_*(p^*(\derived^\bounded(Z))\otimes\mathcal{O}_E(E))
  =\langle \mathcal{A}(-k,-1),\mathcal{A}(-k+1,-1),\ldots,\mathcal{A}(0,-1)\rangle.
\end{equation}
Combining \eqref{eqn:BlowupZX} with \eqref{eq:sod-pbundle}, we have the following semiorthogonal decomposition:
\begin{equation}
  \label{eqn:Blowup1}
  \derived^\bounded(\tilde X)
  =\langle\mathcal{A}(-k,-1),\mathcal{A}(-k+1,-1),\ldots,\mathcal{A}(0,-1),\LLL\tau^*\derived^\bounded(X)\rangle.
\end{equation}
Here all the components except the last one are equivalent to $\derived^\bounded(F)$ via the defining functors.

Similarly, using the fact that $\tilde X$ is also the blowup of $X'$ along $Z'$ and that $Z'$ is a $\PP^{1}$-bundle over $F$, we get a second semiorthogonal decomposition:
\begin{equation}
  \label{eqn:Blowup2}
  \begin{aligned}
    \derived^\bounded(\tilde X)&=
    \langle
      \mathcal{A}(-k, -1),
      \mathcal{A}(-k, 0),
      \mathcal{A}(-k+1,-1),
      \mathcal{A}(-k+1, 0),
      \dots,
      \mathcal{A}(-1, -1),
      \mathcal{A}(-1, 0), \\
      &\qquad\qquad\LLL\tau'^{*}\derived^\bounded(X')
    \rangle.
  \end{aligned}
\end{equation}
Again all the components of \eqref{eqn:Blowup2} except the last one are equivalent to $\derived^\bounded(F)$.

To compare \eqref{eqn:Blowup1} and \eqref{eqn:Blowup2}, we perform a sequence of mutations on \eqref{eqn:Blowup2}.

First we mutate~$\mathcal{A}(-k,0)$ to the left.
The resulting semiorthogonal decomposition is
\begin{equation}
  \begin{aligned}
    \derived^\bounded(\tilde X)&=
    \langle
      \mathcal{A}(-k,-2),
      \mathcal{A}(-k,-1),
      \mathcal{A}(-k+1,-1),
      \mathcal{A}(-k+1,0),
      \ldots,\mathcal{A}(-1,-1),
      \mathcal{A}(-1,0), \\
      &\qquad\qquad\LLL\tau^{\prime*}(\derived^\bounded(X'))
    \rangle.
  \end{aligned}
\end{equation}
Next we mutate $\mathcal{A}(-k,-2)$ to the far right. By \cite[proposition~3.6]{MR1039961}, the resulting decomposition is
\begin{equation}
  \begin{aligned}
    \derived^\bounded(\tilde X)&=
    \langle
      \mathcal{A}(-k,-1),
      \mathcal{A}(-k+1,-1),
      \mathcal{A}(-k+1,0),
      \ldots,
      \mathcal{A}(-1,-1),
      \mathcal{A}(-1,0), \\
      &\qquad\qquad\LLL\tau^{\prime*}\derived^\bounded(X'),
      \serre_{\tilde X}^{-1}\mathcal{A}(-k,-2)
    \rangle,
  \end{aligned}
\end{equation}
where~$\serre_{\tilde X}=-\otimes\omega_{\tilde X}[\dim\tilde X]$ is the Serre functor of~$\derived^\bounded(\tilde X)$. We can easily compute~$\serre_{\tilde X}^{-1}\mathcal{A}(-k,-2)$ as follows
\begin{equation}
  \begin{aligned}
    \serre_{\tilde X}^{-1}\mathcal{A}(-k,-2)
    &=\mathcal{A}(-k, -2)\otimes \omega_{\tilde X}^{\vee}\\
    &=j_*(p^*\circ\pi^*(\derived^\bounded(F))\otimes\mathcal{O}(-k,-2))\otimes\omega_{\tilde X}^\vee\\
    &=j_*(p^*\circ\pi^*(\derived^\bounded(F))\otimes\mathcal{O}(-k,-2)\otimes\omega_{\tilde X}|_{E}^\vee)\\
    &=j_*(p^*\circ\pi^*(\derived^\bounded(F))\otimes\mathcal{O}(-k,-2)\otimes\omega_{E}^\vee\otimes\mathcal{O}_{E}(E))\\
    &=j_*(p^*\circ\pi^*(\derived^\bounded(F))\otimes\mathcal{O}(-k,-2)\otimes\mathcal{O}(k+1, 2)\otimes\mathcal{O}(-1,-1))\\
    &=j_*(p^*\circ\pi^*(\derived^\bounded(F))\otimes\mathcal{O}(0,-1))\\
    &=\mathcal{A}(0, -1).
  \end{aligned}
\end{equation}
The result of the mutation is therefore
\begin{equation}
  \begin{aligned}
    \derived^\bounded(\tilde X)&=
    \langle
      \mathcal{A}(-k,-1),
      \mathcal{A}(-k+1,-1),
      \mathcal{A}(-k+1,0),
      \ldots,
      \mathcal{A}(-1,-1),
      \mathcal{A}(-1,0), \\
      &\qquad\qquad\LLL\tau^{\prime*}(\derived^\bounded(X')),
      \mathcal{A}(0,-1)
    \rangle,
  \end{aligned}
\end{equation}

Then for each~$m=-k+1, \dots, -2$, we right mutate the component~$\mathcal{A}(m, 0)$ through each of the categories~$\mathcal{A}(m,-1),\mathcal{A}(m+1,-1),\ldots,\mathcal{A}(-1,-1)$. Lacking control over the resulting category, we denote the result for now as
\begin{equation}
  \begin{aligned}
    \derived^\bounded(\tilde X)&=
    \langle
      \mathcal{A}(-k,-1),
      \mathcal{A}(-k+1,-1),
      \ldots,
      \mathcal{A}(-1, -1), \\
      &\qquad\qquad\mathcal{B}(-k+1,0),
      \ldots,
      \mathcal{B}(-2,0),
      \mathcal{A}(-1,0),
      \LLL\tau^{\prime*}\derived^\bounded(X'),\mathcal{A}(0,-1)
    \rangle,
  \end{aligned}
\end{equation}
where for all~$-k+1\leq m\leq -2$, we denote
\begin{equation}
  \begin{aligned}
    \mathcal{B}(m,0)
    &\coloneqq\RR_{\langle\mathcal{A}(m,-1),\mathcal{A}(m+1,-1),\dots,\mathcal{A}(-1,-1)\rangle}\mathcal{A}(m, 0) \\
    &=\RR_{\mathcal{A}(-1, -1)}\circ \dots\circ\RR_{\mathcal{A}(m+1, -1)}\circ \RR_{\mathcal{A}(m, -1)}\mathcal{A}(m, 0).
  \end{aligned}
\end{equation}
The identification of the mutation through the subcategory generated by a sequence of subcategories and the composition of mutations follows from e.g.~\cite[lemma~2.2(i)]{MR3987870}.

Finally we right mutate the subcategory~$\langle\mathcal{B}(-k+1,0),\ldots,\mathcal{B}(-2,0),\mathcal{A}(-1,0),\LLL\tau^{\prime*}\derived^\bounded(X')\rangle$ through~$\mathcal{A}(0, -1)$. The result is
\begin{equation}
  \label{eqn:mutated-sod}
  \begin{aligned}
    \derived^\bounded(\tilde X)
    &=
    \langle
      \mathcal{A}(-k,-1),
      \mathcal{A}(-k+1,-1),
      \ldots,
      \mathcal{A}(-1,-1),
      \mathcal{A}(0,-1), \\
      &\qquad\qquad\RR_{\mathcal{A}(0,-1)}\mathcal{B}(-k+1,0),
      \ldots,
      \RR_{\mathcal{A}(0,-1)}\mathcal{B}(-2,0),
      \RR_{\mathcal{A}(0,-1)}\mathcal{A}(-1,0), \\
      &\qquad\qquad\RR_{\mathcal{A}(0,-1)}\LLL\tau^{\prime*}\derived^\bounded(X')
    \rangle.
  \end{aligned}
\end{equation}
As mutations induce equivalences between the original and the mutated components, in the previous decomposition, each component except the last one is equivalent to~$\derived^\bounded(F)$.

Comparing the semiorthogonal decomposition \eqref{eqn:Blowup1} with the mutated semiorthogonal decomposition \eqref{eqn:mutated-sod} obtained in the last step, we deduce the following equality of subcategories of~$\derived^\bounded(\tilde X)$:
\begin{equation}
  \label{eqn:X=FFX'embedded}
  \begin{aligned}
    \LLL\tau^*\derived^\bounded(X)&=
    \langle
      \RR_{\mathcal{A}(0,-1)}\mathcal{B}(-k+1,0),
      \ldots,
      \RR_{\mathcal{A}(0,-1)}\mathcal{B}(-2,0),
      \RR_{\mathcal{A}(0,-1)}\mathcal{A}(-1, 0), \\
      &\qquad\qquad\RR_{\mathcal{A}(0,-1)}\LLL\tau^{\prime*}\derived^\bounded(X')\rangle.
  \end{aligned}
\end{equation}
By construction, each component except the last one is equivalent to~$\derived^\bounded(F)$.

By the fully faithfulness of~$\LLL\tau^*\colon\derived^\bounded(X)\to\derived^\bounded(\tilde X)$, its right adjoint~$\RRR\tau_*\colon\derived^\bounded(\tilde X)\to\derived^\bounded(X)$ induces an equivalence between~$\LLL\tau^*\derived^\bounded(X)$ and~$\derived^\bounded(X)$. Applying this equivalence to~\eqref{eqn:X=FFX'embedded}, we get a semiorthogonal decomposition
\begin{equation}
  \label{eqn:X=FFX'}
  \begin{aligned}
    \derived^\bounded(X)&=
    \langle
      \RRR\tau_*\RR_{\mathcal{A}(0,-1)}\mathcal{B}(-k+1,0),
      \ldots,
      \RRR\tau_*\RR_{\mathcal{A}(0,-1)}\mathcal{B}(-2,0),
      \RRR\tau_*\RR_{\mathcal{A}(0,-1)}\mathcal{A}(-1,0), \\
      &\qquad\qquad\RRR\tau_*\RR_{\mathcal{A}(0,-1)}\LLL\tau^{\prime*}\derived^\bounded(X')
    \rangle.
  \end{aligned}
\end{equation}
It remains to compute the components appearing in this decomposition. We do this via the following trivial lemma, which says that if a functor annihilates a subcategory, then the mutation through this subcategory does not change the image by this functor.

\begin{lemma}
  \label{lem:MutateThruKer}
  Let~$\mathcal{A}$ be an admissible triangulated subcategory of a triangulated subcategory~$\mathcal{T}$. If~$\Psi\colon\mathcal{T}\to \mathcal{T'}$ is a triangulated functor to another triangulated category such that~$\Psi(\mathcal{A})=0$, then
  \begin{equation}
    \Psi(\mathcal{A}^{\perp})=\Psi(\mathcal{T})=\Psi({}^{\perp}\mathcal{A}).
  \end{equation}
  Note that ${}^{\perp}\mathcal{A}=\RR_{\mathcal{A}}\mathcal{A}^{\perp}$.
\end{lemma}

Now, for any integer~$m$, we claim that the functor~$\RRR\tau_*$ annihilates the subcategory~$\mathcal{A}(m,-1)$. Indeed, for any~$\mathcal{E}\in \derived^\bounded(F)$ we have
\begin{equation}
  \begin{aligned}
    \RRR\tau_*\circ j_*(p^*\circ\pi^*(\mathcal{E})\otimes\mathcal{O}(m,-1))
    &=\RRR\tau_*\circ j_*(p^*\circ\pi^*(\mathcal{E})\otimes p^*\mathcal{O}_{\pi}(m)\otimes p^{\prime*}\mathcal{O}_{\pi'}(-1))\\
    &=i_*\circ \RRR p_{*}(p^*\circ \pi^*(\mathcal{E})\otimes \mathcal{O}_\pi(m))\otimes p^{\prime*}\mathcal{O}_{\pi'}(-1))\\
    &=i_*(\pi^*(\mathcal{E})\otimes \mathcal{O}_\pi(m)\otimes\RRR p_*\circ p^{\prime*}(\mathcal{O}_{\pi'}(-1)))
  \end{aligned}
\end{equation}
However by the base change formula we have that
\begin{equation}
  \RRR p_*\circ p^{\prime*}(\mathcal{O}_{\pi'}(-1))=\pi^*\circ\RRR\pi_*'\mathcal{O}_{\pi'}(-1)=0
\end{equation}
which proves the necessary vanishing.

Applying \cref{lem:MutateThruKer} to the functor~$\RRR\tau_*$, we see that a mutation through any category of the form~$\mathcal{A}(m,-1)$ with~$m\in \mathbb{Z}$ does not change the image under~$\RRR\tau_*$. In particular we have
\begin{equation}
  \RRR\tau_*\RR_{\mathcal{A}(0,-1)}\mathcal{B}(m,0)=\RRR\tau_*\mathcal{B}(m,0)=\RRR\tau_*\mathcal{A}(m,0).
\end{equation}
Therefore, the semiorthogonal decomposition \eqref{eqn:X=FFX'} is nothing else but the following:
\begin{equation}
  \label{eqn:X=FFX'NoMutation}
  \derived^\bounded(X)=\langle\RRR\tau_*\mathcal{A}(-k+1,0),\ldots,\RRR\tau_*\mathcal{A}(-2,0),\RRR\tau_*\mathcal{A}(-1,0),\RRR\tau_*\circ\LLL\tau^{\prime*}\derived^\bounded(X')\rangle.
\end{equation}
Finally, since
\begin{equation}
  \begin{aligned}
    \RRR\tau_*\circ j_*(p^*\circ\pi^*\derived^\bounded(F)\otimes\mathcal{O}(n, 0))
    &=i_*\circ\RRR p_*(p^*\circ\pi^*\derived^\bounded(F)\otimes p^*\mathcal{O}_\pi(n))\\
    &=i_*(\pi^*\derived^\bounded(F)\otimes\mathcal{O}_\pi(n))\\
    &=\Phi_n(\derived^\bounded(F)),
  \end{aligned}
\end{equation}
where~$\Phi_n$ is defined in \eqref{equation:Phi-definition}, the decomposition \eqref{eqn:X=FFX'NoMutation} is exactly the one in \cref{theorem:flip}\ref{enumerate:part-2}.

\newpage
\renewcommand*{\bibfont}{\small}
\printbibliography

\emph{Pieter Belmans}, \texttt{pieter.belmans@uni.lu} \\
\textsc{Université du Luxembourg, 6, Avenue de la Fonte, L-4364 Esch-sur-Alzette, Luxembourg}

\emph{Lie Fu}, \texttt{lie.fu@math.unistra.fr} \\
\textsc{Université de Strasbourg, Institut de Recherche Mathématique Avancée (IRMA), 7 rue René-Descartes, 67084 Strasbourg Cedex, France}

\emph{Theo Raedschelders}, \texttt{theo.raedschelders@vub.be} \\
\textsc{Vrije Universiteit Brussel, Pleinlaan 2, 1050 Brussels, Belgium}


\end{document}